\documentclass[reqno]{amsart}
\usepackage{amsfonts,amsmath,amssymb}

\numberwithin{equation}{section}

\newtheorem{theorem}{Theorem}[section]
\newtheorem{lemma}{Lemma}[section]

\newtheorem{corollary}{Corollary}[section]
\newtheorem{proposition}{Proposition}[section]

\begin{document}

\title[Cubic fourth-order NLS]
{The cubic fourth-order Schr\"odinger equation}
\author{Benoit Pausader}
\address{Department of Mathematics,
University of Cergy-Pontoise,
CNRS UMR 8088,
2, avenue Adolphe Chauvin,
95302 CERGY-PONTOISE cedex,
France}
\email{Benoit.Pausader@math.u-cergy.fr}

\begin{abstract}
Fourth-order Schr\"odinger equations have been introduced by 
Karpman and Shagalov to take
into account the role of small fourth-order dispersion terms in
the propagation of intense laser beams in a bulk medium with Kerr
nonlinearity. In this paper we investigate the cubic defocusing fourth order Schr\"odinger equation 
$$i\partial_tu + \Delta^2u + \vert u\vert^2u=0$$
in arbitrary 
space dimension $\mathbb{R}^n$ for arbitrary initial data. We prove that the equation is globally 
well-posed when $n \le 8$ and ill-posed when $n \ge 9$, with the additional important 
information that scattering holds true when $5 \le n \le 8$. 
\end{abstract}

\maketitle

\section{Introduction}\label{Section--Intro}

Fourth-order Schr\"odinger equations have been introduced by
Karpman \cite{Kar} and Karpman and Shagalov \cite{KarSha} to take
into account the role of small fourth-order dispersion terms in
the propagation of intense laser beams in a bulk medium with Kerr
nonlinearity. Such fourth-order Schr\"odinger equations have been 
studied from the mathematical viewpoint in Fibich,
Ilan and Papanicolaou \cite{FibIlaPap} who describe various properties
of the equation in the subcritical regime, with part of their
analysis relying on very interesting numerical developments. Related references are by 
Ben-Artzi, Koch, and Saut \cite{BenKocSau} who gave sharp 
dispersive estimates for the biharmonic Schr\"odinger
operator, Guo
and Wang \cite{GuoWan} who proved global well-posedness and scattering
in $H^s$ for small data, Hao, Hsiao and Wang \cite{HaoHsiWan1,
HaoHsiWan2} who discussed the Cauchy problem in a high-regularity
setting, and Segata \cite{Seg3} who proved scattering in the case the
space dimension is one. We refer also to Pausader \cite{Pau1,Pau1Plus} where 
the energy critical case for radially symmetrical initial data 
is discussed. The 
defocusing case like in \eqref{CubicFOS} below is discussed in Pausader \cite{Pau1} for radially symmetrical initial data. The focusing case, 
following the beautiful results of Kenig and Merle \cite{KenMer,KenMer2},  is settled in Pausader \cite{Pau1Plus} still for radially symmetrical initial data.

\medskip We focus in this paper 
on the study of the initial value problem for the cubic fourth-order defocusing equation 
in arbitrary space dimension $\mathbb{R}^n$, $n \ge 1$, without assuming radial symmetry for the 
intial data. The equation is written as
\begin{equation}\label{CubicFOS}
 i\partial_tu + \Delta^2u + \vert u\vert^2u=0\hskip.1cm ,
\end{equation}
where $u = I\times\mathbb{R}^n \to \mathbb{C}$ is a complex valued function, and 
$u_{\vert t=0} = u_0$ is in $H^2$, the space of $L^2$ functions whose 
first and second derivatives are in $L^2$. The equation is critical when $n = 8$ because of the criticality 
of the Sobolev embedding $H^2 \subset L^4$ in this dimension, and it enjoys rescaling invariance 
leaving the energy and $\dot H^2$-norm unchanged. Let $\mathcal{S}$ be the space of Schwartz functions.
The theorem we prove in this paper provides a complete picture of global well-posedness 
for \eqref{CubicFOS}. It is stated as follows.

\begin{theorem}\label{MainThm} Assume $1 \le n \le 8$. Then for any $u_0\in H^2$ there exists a global solution $u\in C(\mathbb{R},H^2)$ of 
\eqref{CubicFOS} with initial data $u(0)=u_0$. Moreover, for any $t\in\mathbb{R}$, the mapping $u(0)\mapsto u(t)$ is 
analytic from $H^2$ into itself. On the contrary, if $n \ge 9$ then the Cauchy problem for \eqref{CubicFOS} is ill-posed in $H^2$ in the sense 
that for any $\varepsilon > 0$, 
there exist $u_0 \in \mathcal{S}$, $t_\varepsilon \in (0,\varepsilon)$, and 
$u \in C\left([0,\varepsilon],H^2\right)$ a solution of \eqref{CubicFOS} with initial data $u_0$ 
such that $\Vert u_0\Vert_{H^2} < \varepsilon$ while $\Vert u(t_\varepsilon)\Vert_{H^2} > \varepsilon^{-1}$. 
Besides, if $5 \le n \le 8$, then scattering holds true in $H^2$ for \eqref{CubicFOS} and the scattering operator 
is analytic.
\end{theorem}

The fourth-order dispersion scaling property leads to the heuristic that smooth solutions of the free homogeneous equation 
have their $L^\infty$ norm which decays like $t^{-\frac{n}{4}}$. However, the situation is not so transparent 
and all frequency parts of the function have their $L^\infty$-norm that decays much faster, like $t^{-\frac{n}{2}}$, 
but at a rate which depends on the frequency. 
Uniformly, the rate of decay $t^{-\frac{n}{4}}$ is the best possible, but it is not optimal when the solution is localized 
in frequency. As one will see, there are various differences between the dispersion
behaviors of second-order Schr\"odinger equations and of \eqref{CubicFOS}.

\medskip Our paper is organised as follows.
We fix notations in Section \ref{Section--Notations} and recall preliminary results from 
Pausader \cite{Pau1} in Section \ref{Section-PrelRes}. In Section \ref{Section-Ill-posed}, we prove 
that the Cauchy problem is ill-posed when $n \ge 9$. 
In order to do so we use a low-dispersion regime argument which was essentially given in Christ, Colliander and Tao \cite{ChrColTao}. We also refer to Lebeau \cite{Leb2,Leb}, Alazard and Carles \cite{AlaCar}, Carles \cite{Car} 
and Thomann \cite{Tho1,Tho2} for other results in different settings. Starting from Section \ref{Section-3Scenarii} 
we focus on the energy-critical case, and so on the $n=8$ part of our theorem (the equation is subcritical when $n \le 7$). We prove in Section
\ref{Section-3Scenarii}, using important ideas of concentration compactness developed in Kenig and Merle \cite{KenMer} and Killip, Tao and Visan \cite{KilTaoVis}, 
that any failure of global wellposedness implies the existence of some special solutions satisfying three possible scenarii. 
The remaining part of the analysis consists in excluding these hypothetical special solutions working at the level of $\dot{H}^2$-solutions. The first scenario is that there is a self-similar-like solution. 
It is not consistent with conservation of energy, conservation of local mass and compactness up to rescaling. We exclude this scenario in Section
\ref{Section-Self-Similar}. The two other scenarii are that there is a soliton-like solution or that there is a low-to-high cascade-like solution. 
In these two scenarii the solution is away from the $L^2$-like region, namely we have that $h\le 1$ with respect to the notation of 
Theorem \ref{3S-H2CritThm3S}. 
We use this to prove an interaction Morawetz estimate in Sections
\ref{Section-IME} and \ref{Section-FLIME}, following previous analysis from Colliander, Keel, Staffilani, Takaoka and Tao \cite{ColKeeStaTakTao}, Ryckman and Visan \cite{RycVis} and Visan \cite{Vis}.
The estimate we prove is not an a priori estimate. A major difficulty is that the estimate scales like the $\dot{H}^\frac{1}{4}$-norm and thus creates a $7/4$-difference in scaling with the
$\dot{H}^2$-norm control we have.
In Section \ref{Section-Soliton}, we exclude soliton-like solution by proving that it is not consistent with the frequency-localized interaction Morawetz estimates and compactness up to rescaling. The last scenario is excluded in Section \ref{Section-Cascade} by proving that any low-to-high-like solution has an unexpected $L^2$-regularity. Then, conservation of $L^2$-norm, frequency-localized interaction Morawetz estimates and conservation of energy allows us to exclude this existence of low-to-high cascade-like cascade solutions. 
Finally, in Section \ref{Section-Scattering}, we prove the scattering part of Theorem \ref{MainThm}.

\medskip As a remark, with the arguments we develop here and adaptations of the analysis in Visan \cite{Vis},
global well-posedness and scattering in Theorem \ref{MainThm} continue to hold true 
when $n\ge 8$ and the cubic nonlinearity is replaced by the $n$-dimensional energy-critical nonlinearity with total power $(n+4)/(n-4)$.  We also refer to Miao, Xu and Zhao \cite{MiaXuZha} for another proof in high dimensions $n\ge 9$ following previous work by Killip and Visan \cite{KilVis}.
For radially symmetrical data, see Pausader \cite{Pau1}, this is also true in any dimension $n \ge 5$.


\section{Notations}\label{Section--Notations}

We fix notations we use throughout the paper. In what follows, we write $A\lesssim B$ to signify that there exists a constant $C$ depending only on $n$ such that $A\le C B$. When the constant $C$ depends on other parameters, we indicate this by a subscript, for exemple, $A\lesssim_u B$ means that the constant may depend on $u$. Similar notations hold for $\gtrsim$.  Similarly we write $A\simeq B$ when $A\lesssim B\lesssim A$.

\medskip

We let $L^q=L^q(\mathbb{R}^n)$ be
the usual Lebesgue spaces, and $L^r(I,L^q)$ be the space of
measurable functions from an interval $I\subset\mathbb{R}$ to
$L^q$ whose $L^r(I,L^q)$ norm is finite, where
\begin{equation*}
\Vert u\Vert_{L^r(I,L^q)}=\left(\int_I\Vert
u(t)\Vert_{L^q}^rdt\right)^\frac{1}{r}.
\end{equation*}
When there is no risk of confusion we may write $L^qL^r$ instead of $L^q(I,L^r)$.
Two important conserved quantities of equation \eqref{CubicFOS}
are the mass and the energy. The mass is defined by
\begin{equation}\label{DefinitionOfMass}
 M(u)=\int_{\mathbb{R}^n}\vert u(x)\vert^2dx
\end{equation}
and the energy is defined by
\begin{equation}\label{DefinitionOfEnergy}
 E(u)=\int_{\mathbb{R}^n}\left(\frac{\vert\Delta u(x)\vert^2}{2}+\frac{\vert u(x)\vert^4}{4}\right)dx.
\end{equation}
In what follows we let $\mathcal{F}f=\hat{f}$ be the Fourier transform of $f$
given by
\begin{equation*}\label{DefFourier}
\hat{f}(\xi)
=\frac{1}{(2\pi)^\frac{n}{2}}\int_{\mathbb{R}^n}f(y)e^{i\left<y,\xi\right>}dy
\end{equation*}
for all $\xi \in \mathbb{R}^n$. The biharmonic Schr\"odinger
semigroup is defined for any tempered distribution $g$ by
\begin{equation}\label{DefFOSSemigroup}
e^{it\Delta^2}g=\mathcal{F}^{-1}e^{it\vert\xi\vert^4}\mathcal{F}g.
\end{equation}
Let $\psi\in C^\infty_c(\mathbb{R}^n)$ be supported in the ball
$B(0,2)$, and such that $\psi=1$ in $B(0,1)$. For any dyadic
number $N=2^k,k\in\mathbb{Z}$, we define the following
Littlewood-Paley operators:
\begin{equation}\label{DefLitPalOp}
\begin{split}
&\widehat{P_{\leq N}f}(\xi)=\psi(\xi/N)\hat{f}(\xi),\\
&\widehat{P_{>N}f}(\xi)=(1-\psi(\xi/N))\hat{f}(\xi),\\
&\widehat{P_Nf}(\xi)=\left(\psi(\xi/N)-\psi(2\xi/N)\right)\hat{f}(\xi).
\end{split}
\end{equation}
Similarly we define $P_{<N}$ and $P_{\ge N}$ by the equations
$$P_{<N} = P_{\leq N}-P_N\hskip.2cm\hbox{and}\hskip.2cm P_{\ge N} = P_{> N} + P_N$$
These operators commute one with another. They also commute with
derivative operators and with the semigroup
$e^{it\Delta^2}$. In addition they are
self-adjoint and bounded on $L^p$ for all $1\le p\le\infty$.
Moreover, they enjoy the following Bernstein property:
\begin{equation}\label{BernSobProp}
\begin{split}
&\hskip.8cm \Vert P_{\ge N}f\Vert_{L^p}\lesssim_s
N^{-s}\Vert\vert\nabla\vert^sP_{\geq N}f\Vert_{L^p}\lesssim_s N^{-s}\Vert\vert\nabla\vert^sf\Vert_{L^p}\\
&\hskip.8cm\Vert\vert\nabla\vert^sP_{\le N}f\Vert_{L^p}\lesssim_s N^s\Vert P_{\le
N} f\Vert_{L^p}
\lesssim_s N^s\Vert f\Vert_{L^p}\\
&\hskip.8cm\Vert \vert\nabla\vert^{\pm s}P_Nf\Vert_{L^p}\lesssim_s
N^{\pm s}\Vert P_Nf\Vert_{L^p}\lesssim_s N^{\pm s}\Vert f\Vert_{L^p}\\
\end{split}
\end{equation}
for all $s\ge 0$, and all $1\le p\le\infty$, independently of $f$, $N$, and $p$, where
$\vert\nabla\vert^s$ is the classical fractional differentiation
operator. We refer to Tao \cite{TaoBook} for more details. Given
$a\geq 1$, we let $a^\prime$ be the conjugate of $a$, so that
$\frac{1}{a}+\frac{1}{a^\prime}=1$.

\medskip

Several norms have to be considered in the analysis of the critical case of
 \eqref{CubicFOS}. For $I \subset \mathbb{R}$ an interval, they are defined as
\begin{equation}\label{N-DefinitionOfNorms}
 \begin{split}
  \Vert u\Vert_{M(I)}&=\Vert\Delta u\Vert_{L^{\frac{2(n+4)}{n-4}}(I,L^{\frac{2n(n+4)}{n^2+16}})}\hskip.1cm,\\
  \Vert u\Vert_{W(I)}&=\Vert\nabla u\Vert_{L^{\frac{2(n+4)}{n-4}}(I,L^{\frac{2n(n+4)}{n^2-2n+8}})}\hskip.1cm,\\
  \Vert u\Vert_{Z(I)}&=\Vert u\Vert_{L^{\frac{2(n+4)}{n-4}}(I,L^{\frac{2(n+4)}{n-4}})}\hskip.1cm,\hskip.1cm\text{and}\\
  \Vert u\Vert_{N(I)}&=\Vert\nabla u\Vert_{L^2(I,L^\frac{2n}{n+2})}.
 \end{split}
\end{equation}
Accordingly, we let $M(\mathbb{R})$ be the completion of
$\mathcal{S}(\mathbb{R}^{n+1})$ with the norm
$\Vert\cdot\Vert_{M(\mathbb{R})}$, and $M(I)$ be the set consisting of
the restrictions to $I$ of functions in $M(\mathbb{R})$. We adopt similar definitions for
$W$, $Z$, and $N$.
We also need the following stronger norms in order to fully exploit the Strichartz estimates in Section \ref{Section-PrelRes}.
Following standard notations, we say that a pair $(q,r)$ is Schr\"odinger-admissible, for short S-admissible, if
$2\le q,r\le\infty,$ $(q,r,n)\ne(2,\infty,2)$, and
\begin{equation}\label{Not-SE-DefinitionOfSadmissible}
 \frac{2}{q}+\frac{n}{r}=\frac{n}{2}.
\end{equation}
We define the full Strichartz norm of regularity $s$ by
\begin{equation}\label{Not-CompleteNorm}
\Vert u\Vert_{\dot{S}^s(I)}=\sup_{(a,b)}\left(\sum_NN^{2s+\frac{4}{a}}\Vert P_Nu\Vert_{L^a(I,L^b)}^2\right)^\frac{1}{2},
\end{equation}
where the supremum is taken over all $S$-admissible pairs $(a,b)$ as in \eqref{Not-SE-DefinitionOfSadmissible}, $s\in\mathbb{R}$ and $I\subset\mathbb{R}$ is an interval.
We also define the dual norm,
\begin{equation}\label{Not-DualNorm}
 \Vert h\Vert_{\dot{\bar{S}}^s(I)}=\inf_{(a,b)}\left(\sum_NN^{2s-\frac{4}{a}}\Vert P_Nh\Vert_{L^{a^\prime}(I,L^{b^\prime})}^2\right)^\frac{1}{2}
\end{equation}
where again, the infimum is taken over all $S$-admissible pairs $(a,b)$ as in \eqref{Not-SE-DefinitionOfSadmissible}, $s\in\mathbb{R}$, and $I$ is an interval. We let $\dot{S}^s(I)$ be the set of tempered distributions of finite $\dot{S}^s(I)$-norm.
Finally, for a product $\pi=\Pi_ia_i$, we use the notation $\mathcal{O}(\pi)$ to denote an expression which is schematically like $\pi$, i.e. that is a finite combination of products $\pi^\prime=\Pi_ib_i$ where in each $\pi^\prime$, each $b_i$ stands for $a_i$ or for $\bar{a}_i$.

\medskip As a remark, if $n = 8$, then there is a rescaling invariance rule for \eqref{CubicFOS} given by
\begin{equation}\label{I-RescalingLaw}
u\mapsto \tau_{(h,t_0,x_0)}u= h^2u(h^4(t-t_0),h(x-x_0))
\end{equation}
which sends a solution of \eqref{CubicFOS} with initial data $u(0)=u_0$ to another solution with data at time
$t=t_0$ given by
\begin{equation}\label{I-DefOfG}
 g_{(h,x_0)}u_0=h^2u_0(h(x-x_0)),
\end{equation}
and which leaves the energy and $\dot{H}^2$-norm unchanged: 
$$E\left(\tau_{(h,t_0,x_0)}u\right) = E\left(u\right)\hskip.2cm\hbox{and}\hskip.2cm \left\Vert g_{(h,x_0)}u_0\right\Vert_{\dot H^2} = 
\left\Vert u_0\right\Vert_{\dot H^2}$$
for all $u_0, u, h, t_0, x_0$. The associated loss of compactness makes 
that \eqref{CubicFOS} is particularly difficult to handle in the critical dimension $n=8$. In the radially symmetrical case the 
difficulty was overcome in Pausader \cite{Pau1}. We prove here that we can get rid of the 
radially symmetrical assumption. 


\section{Preliminary results}\label{Section-PrelRes}

We recall results from Pausader \cite{Pau1}. We refer to Pausader \cite{Pau1} for their proof. 
A first result from Pausader \cite{Pau1} is that the following fundamental Strichartz-type estimates hold true. 
Note that these estimates, because of the gain of derivatives, contradict the Galilean invariance one could have expected for the fourth order 
Schr\"odinger equation.

\begin{proposition}\label{StrichPropEst}
Let $u\in C(I,H^{-4})$ be a solution of
\begin{equation}\label{Not-LinSol}
 i\partial_tu+\Delta^2u+h=0,
\end{equation}
and $u(0)=u_0$.
Then, for any $S$-admissible pairs $(q,r)$ and $(a,b)$ as in \eqref{Not-SE-DefinitionOfSadmissible}, and any
$s\in\mathbb{R}$,
\begin{equation}\label{Not-SE-StrichartzEstimatesWithGainI}
\Vert \vert\nabla\vert^s u\Vert_{L^q(I,L^r)}\lesssim \left(\Vert
\vert\nabla\vert^{s-\frac{2}{q}}
u_0\Vert_{L^2}+\Vert\vert\nabla\vert^{s-\frac{2}{q}-\frac{2}{a}}h\Vert_{L^{a^\prime}(I,L^{b^\prime})}\right)
\end{equation}
whenever the right hand side in
\eqref{Not-SE-StrichartzEstimatesWithGainI} is finite.
\end{proposition}

A consequence of the Strichartz estimates \eqref{Not-SE-StrichartzEstimatesWithGainI} and of the
commutation properties of the linear propagator $e^{it\Delta^2}$ is the following estimate, for any solution $u$ as above:
\begin{equation}\label{Not-SE-StrichartzEstimatesWithGainII}
\begin{split}
\Vert u\Vert_{\dot{S}^s(I)}&\lesssim \Vert u_0\Vert_{\dot{H}^s}+\Vert h\Vert_{\dot{\bar{S}}^s(I)}\\
&\lesssim\Vert u_0\Vert_{\dot{H}^s}+\Vert \vert\nabla\vert^{s-\frac{2}{a}}h\Vert_{L^{a^\prime}(I,L^{b^\prime})},
\end{split}
\end{equation}
where $(a,b)$ is an $S$-admissible pair as in \eqref{Not-SE-DefinitionOfSadmissible}, and the norms are defined in \eqref{Not-CompleteNorm} and \eqref{Not-DualNorm} above.
 A preliminary version of \eqref{Not-SE-StrichartzEstimatesWithGainI} was obtained in Kenig, Ponce and Vega \cite{KenPonVeg}.  Let $u\in C(I,\dot{H}^2)$ be
defined on some interval $I$ such that $0\in I$ and such that $u\in
L^3_{loc}(I\times\mathbb{R}^n)$. We say that $u$ is a solution of \eqref{CubicFOS}
provided that the following equality holds in the sense of
tempered distributions for all times:
\begin{equation}\label{DefinitionOfSolution}
 u(t)=e^{it\Delta^2}u_0+i\int_0^te^{i(t-s)\Delta^2}\left(\vert u\vert^2u\right)(s)ds.
\end{equation}
Note that, by Strichartz estimates, if $u_0\in L^2$ and $\vert u\vert^2u\in
L_{loc}^1(I,L^2)$, then \eqref{DefinitionOfSolution} is equivalent to the fact that $u$
solves \eqref{CubicFOS} in $H^{-4}$
with $u(0)=u_0$.

\medskip

The following Propositions \ref{Not-BCC-ClaimOfLE} and \ref{Not-S-StabProp2}, still from Pausader \cite{Pau1}, are important for the energy-critical case $n=8$.
Proposition \ref{Not-BCC-ClaimOfLE} settles the question of local well-posedness. Proposition \ref{Not-S-StabProp2} settles the question of stability. 

\begin{proposition}\label{Not-BCC-ClaimOfLE} Let $n=8$.
 There exists $\delta>0$ such that for any initial data $u_0\in \dot{H}^2$, and any interval $I=[0,T]$, if
 \begin{equation}\label{Not-BCC-ConditionForLEintheCriticalCase}
  \Vert e^{it\Delta^2}u_0\Vert_{W(I)}<\delta,
 \end{equation}
then there exists a unique solution $u\in
C(I,\dot{H}^2)$ of \eqref{CubicFOS} with
initial data $u_0$. This solution has conserved energy, and satisfies $u\in \dot{S}^2(I)$. Moreover,
 \begin{equation}\label{Not-BCC-ControlOfNorms}
 \begin{split}
 &\Vert u\Vert_{\dot{S}^2(I)}\lesssim
 \Vert u_0\Vert_{\dot{H}^2}+\delta^3,
 \end{split}
 \end{equation}
 and if $u_0\in H^2$, then $u\in \dot{S}^0(I)\cap \dot{S}^2(I)$,
 \begin{equation*}
  \Vert u\Vert_{\dot{S}^0(I)}\lesssim \Vert u_0\Vert_{L^2},
 \end{equation*}
 and $u$ has conserved mass.
 Besides, in this case, the solution depends continuously on the initial data
 in the sense that there exists $\delta_0$, depending on $\delta$, such
 that, for any $\delta_1\in (0,\delta_0)$, if $\Vert v_0-u_0 \Vert_{H^2}\le \delta_1$, and if we let $v$ be
 the local solution of \eqref{CubicFOS} with
 initial data $v_0$, then $v$ is defined on $I$ and
 $\Vert u-v\Vert_{\dot{S}^0(I)} \lesssim \delta_1.$
\end{proposition}

In addition to Proposition \ref{Not-BCC-ClaimOfLE} we also have Proposition \ref{Not-S-StabProp2}.

\begin{proposition}\label{Not-S-StabProp2} Let $n =8$,
$I \subset \mathbb{R}$ be a compact time interval such that
$0 \in I$, and $\tilde{u}$ be an approximate solution
of \eqref{CubicFOS} in the sense that
\begin{equation}\label{Not-S-AlmostSolution}
i\partial_t\tilde{u}+\Delta^2\tilde{u}+\vert
\tilde{u}\vert^2\tilde{u}=e
\end{equation}
for some $e \in N(I)$. Assume that  $\Vert \tilde{u}\Vert_{Z(I)} <
+\infty$ and $\Vert \tilde{u}\Vert_{L^\infty(I,\dot{H}^2)} <
+\infty$. There exists $\delta_0 > 0$,
$\delta_0 = \delta_0(\Lambda,\Vert \tilde{u}\Vert_{Z(I)},\Vert
\tilde{u}\Vert_{L^\infty(I,\dot{H}^2)})$,
 such that if $\Vert e\Vert_{N(I)}\le \delta$, and $u_0 \in \dot{H}^2$ satisfies
\begin{equation}\label{Not-S-Stab2Hyp2}
\Vert \tilde{u}(0)-u_0\Vert_{\dot{H}^2}\le
\Lambda\hskip.2cm\hbox{and}\hskip.2cm \Vert
e^{it\Delta^2}\left(\tilde{u}(0)-u_0\right)\Vert_{W(I)}\le\delta
\end{equation}
for some $\delta \in (0,\delta_0]$, then there exists $u\in C(I,\dot{H}^2)$ a solution
of \eqref{CubicFOS} such that
$u(0)=u_0$. Moreover, $u$ satisfies
\begin{equation}\label{Not-S-Stab2CCl}
\begin{split}
&\Vert u-\tilde{u}\Vert_{W(I)}\le
C\delta\hskip.1cm ,\\
&\Vert u-\tilde{u}\Vert_{\dot{S}^2}\le
C\left(\Lambda+\delta\right)\hskip.1cm ,\hskip.1cm\hbox{and}\\
&\Vert u\Vert_{\dot{S}^2}\le C,
\end{split}
\end{equation}
where $C = C(\Lambda,\Vert
\tilde{u}\Vert_{Z(I)},\Vert
\tilde{u}\Vert_{L^\infty(I,\dot{H}^2)})$ is a nondecreasing function of its arguments.
\end{proposition}

In our analysis, we need to consider $\dot{H}^2$-solutions. These solutions do not satisfy conservation of mass. However the next proposition shows that there is still something remaining from that conservation law for these solutions.
Proposition \ref{Not-AlmostConservationOfMass} shows that the local mass of
a solution of  \eqref{CubicFOS} varies slowly in time provided that
the radius $R$ is sufficiently large.
We define the local
mass $M\left(u,B(x_0,R)\right)$ over the ball $B(x_0,R)$ of a
function $u\in L^2_{loc}$ by
\begin{equation}\label{Not-DefLocalMass}
M\left(u,B(x_0,R)\right)=\int_{\mathbb{R}^n}\vert u(x)\vert^2\psi^4\left((x-x_0)/R\right)dx,
\end{equation}
where, $\psi$ is as in \eqref{DefLitPalOp}. Proposition \ref{Not-AlmostConservationOfMass} from Pausader \cite{Pau1}, states as follows.

\begin{proposition}\label{Not-AlmostConservationOfMass}
Let $n\ge 5$, and $u\in C(I,\dot{H}^2)$ be a solution of
 \eqref{CubicFOS}.
Then we have that
 \begin{equation}\label{Not-AlmostConservationOfLocalMassEstimate}
\left\vert\partial_tM\left(u(t),B(x_0,R)\right)\right\vert\lesssim
\frac{E(u)^\frac{3}{4}}{R}M\left(u(t),B(x_0,R)\right)^\frac{1}{4}
\end{equation}
for all $t \in I$.
\end{proposition}

We refer to Pausader \cite{Pau1} for a proof of the above propositions.


\section{Ill-posedness results}\label{Section-Ill-posed}

In this section we use a quantitative analysis of the small dispersion regime to prove ill-posedness results for the cubic equation when $n>8$. The idea is that now the equation is supercritical with repect to the regularity-setting in which we work, namely $H^2$. Hence one can always use rescaling arguments to make any ``separation-mechanism'' between two different solutions happen sooner and sooner while making the $H^2$-norm smaller and smaller. It remains then to find two solutions whose distance goes to $\infty$ as time evolves. To achieve this, we follow the proof in Christ, Colliander and Tao \cite{ChrColTao} by considering
the small dispersion regime. See also Lebeau \cite{Leb2,Leb} for previous results, and Alazard and Carles \cite{AlaCar}, Carles \cite{Car} and Thomann \cite{Tho1,Tho2} for instability results in different contexts.

\medskip

Before we prove our theorem, we need the following lemma concerning the small dispersion regime.

\begin{lemma}\label{IP-SmallDispLemma}
Let $k>n/2$. Then, for any $\phi\in\mathcal{S}$, there exists $c>0$ such that for any $\nu\in(0,1)$, there exists a unique solution $w^{\nu}\in C([-T,T],H^k)$ of the problem
\begin{equation}\label{IP-SmallDispEqt}
i\partial_tw+\nu^4\Delta^2w+\vert w\vert^2w=0
\end{equation}
with initial data $w^{\nu}(0)=\phi$, where $T=c\vert\log\nu\vert^c$. Besides, the solution satisfies $w^\nu\in C([-T,T],H^p)$ for any $p$, and
\begin{equation}\label{IP-SmallDispQuantitativeEst}
\Vert w^{\nu}-w^{0}\Vert_{L^\infty([-T,T],H^k)}\lesssim_{\phi,k}\nu^3,
\end{equation}
where
\begin{equation}\label{IP-DefOfW0}
w^{0}(t,x)=\phi(x)\exp\left(i\vert \phi(x)\vert^2t\right)
\end{equation}
is a solution of the ODE formally obtained by setting $\nu=0$ in \eqref{IP-SmallDispEqt}.
\end{lemma}

\begin{proof}
Letting $u=w^{\nu}-w^{0}$, we see that $u$ solves the Cauchy problem
\begin{equation}\label{IP-SmallDsipCP}
\begin{split}
&i\partial_tu+\nu^4\Delta^2u=\nu^4\Delta^2 w^{0}+\vert w^0\vert^2w^0-\vert w^0+u\vert^2(w^0+u)
\end{split}
\end{equation}
with $u(0)=0$. Let $k>n/2$ be given. Since $w^0\in C^\infty(\mathcal{S})$, standard developments ensure that there exists a unique solution $u\in C([-t,t],H^k)$ to \eqref{IP-SmallDsipCP}, and that $u$ can be continued as long as $\Vert u\Vert_{H^k}$ remains bounded. Besides, $u\in C([-t,t],H^p)$ for any $p\ge 0$ (in the sense that $t$ does not depend on $p$). Consequently, it suffices to prove that there exists $c>0$ such that for any $s<c\vert\log\nu\vert^c$, we have that $\Vert u(s)\Vert_{H^k}\le \nu^3$.
Now, taking derivatives $\partial^\alpha$ of equation \eqref{IP-SmallDsipCP}, multiplying by $\partial^\alpha\bar{u}$, taking the imaginary part and integrating, for all $\alpha$ such that $\vert \alpha\vert\le k$, we get that
\begin{equation}\label{IP-EnergyEstForLE}
\begin{split}
&\partial_s\Vert u(s)\Vert_{H^k}^2\\
&\lesssim \Vert u\Vert_{H^k}\left(\nu^4\Vert \Delta^2w^0(s)\Vert_{H^k}+\Vert \vert w^0+u\vert^2(w^0+u)-\vert w^0\vert^2w^0\Vert_{H^k}\right).
\end{split}
\end{equation}
By \eqref{IP-DefOfW0} we see that, for $p\ge 0$,
\begin{equation}\label{IP-EnergyEstForLEEst1}
\Vert w^0\Vert_{H^{p}}\lesssim_{\phi,p}t^{p}.
\end{equation}
Independently, since $H^k$ is an algebra, we get that
\begin{equation}\label{IP-EnergyEstForLEEst2}
\begin{split}
\Vert \vert w^0+u\vert^2(w^0+u)-\vert w^0\vert^2w^0\Vert_{H^k}&\lesssim \sum_{j=0}^2\Vert \mathcal{O}\left(\left(w^0\right)^ju^{3-j}\right)\Vert_{H^k}\\
&\lesssim \Vert u\Vert_{H^k}\left(1+\Vert w^0\Vert_{H^k}+\Vert u\Vert_{H^k}\right)^2.
\end{split}
\end{equation}
Now, using \eqref{IP-EnergyEstForLE}--\eqref{IP-EnergyEstForLEEst2}, we see that, in the sense of distributions,
\begin{equation}\label{IP-EnergyEstForLECCl}
\partial_s\Vert u(s)\Vert_{H^k}\lesssim_{\phi,k}\nu^4\left(1+\vert s\vert^{k+4}\right)+\Vert u(s)\Vert_{H^k}\left(1+\vert s\vert^{k}+\Vert u(s)\Vert_{H^k}\right)^2.
\end{equation}
An application of Gromwall's lemma gives the bound
\begin{equation}\label{IP-EnergyEstForLELast}
\Vert u(s)\Vert\lesssim_{k,\phi}\nu^4\exp\left(C\left(1+\vert s\vert^C\right)\right)
\end{equation}
for all $s$ such that $\Vert u(s)\Vert_{H^k}\le 1$. By \eqref{IP-EnergyEstForLELast} we see that $\Vert u(s)\Vert_{H^k}\le 1$ holds for all times $\vert s\vert\le c\vert\log \nu\vert^c$, $c>0$ sufficiently small. This gives \eqref{IP-SmallDispQuantitativeEst} and finishes the proof of Lemma \ref{IP-SmallDispLemma}.
\end{proof}

Now, we are in position to prove the main theorem of this section which states that the flow map $u_0\mapsto u(t)$, from $H^2$ into $H^2$ which maps the initial data to the associated solution fails to be continuous at $0$. As a remark, note that \eqref{IP-IllPosednessResult} is false when $n\le 8$ since the $H^2$-norm controls the energy.

\begin{theorem}\label{IP-IPTheorem}
Let $n>8$. Given $\varepsilon>0$, there exists a solution $u\in C([0,\varepsilon],H^2)$ such that
\begin{equation}\label{IP-IllPosednessResult}
\begin{split}
\Vert u(0)\Vert_{H^2}<\varepsilon\hskip.3cm\hbox{and}\hskip.3cm\Vert u(t_\varepsilon)\Vert_{H^2}>\varepsilon^{-1},
\end{split}
\end{equation}
for some $t_\varepsilon\in (0,\varepsilon)$. Besides, we can choose $u$ such that $u(0)\in\mathcal{S}$ and $u\in C([0,\varepsilon],H^k)$ for any $k>0$.
\end{theorem}

\begin{proof}[Proof of Theorem \ref{IP-IPTheorem}]
For $\phi\in\mathcal{S}$ and $\nu\in (0,1]$, we let $w^{\nu}$ be the solution of equation \eqref{IP-SmallDispEqt} with initial data $w^{\nu}(0)=\phi$. By Lemma \ref{IP-SmallDispLemma}, we see that for $\vert s\vert\le c\vert\log \nu\vert^c$, \eqref{IP-SmallDispQuantitativeEst} holds true for $w^{0}$ as in \eqref{IP-DefOfW0}. Now, for $\lambda\in (0,\infty)$, we let
\begin{equation}\label{IP-DefOfSol}
 u^{(\nu,\lambda)}(t,x)=\lambda^2w^{\nu}(\lambda^4t,\lambda\nu x).
\end{equation}
Then $u^{(\nu,\lambda)}$ solves \eqref{CubicFOS} with initial data $u^{(\nu,\lambda)}(0,x)=\lambda^2\phi(\lambda\nu x)$.
A simple calculation gives
\begin{equation}\label{IP-CompOFHSNorm}
 \begin{split}
  \Vert u^{(\nu,\lambda)}(0)\Vert_{H^2}^2&=\frac{\lambda^4}{\left(2\pi\right)^n}\left(\lambda\nu\right)^{-2n}\int_{\mathbb{R}^n}\vert\hat{\phi}(\xi/(\lambda\nu))\vert^2(1+\vert\xi\vert^2)^2d\xi\\
&\lesssim\lambda^4\left(\lambda\nu\right)^{-n}\Big(\int_{\mathbb{R}^n}\vert\hat{\phi}(\eta)\vert^2\vert\lambda\nu\eta\vert^{4}d\eta+\int_{\mathbb{R}^n}\vert\hat{\phi}(\eta)\vert^2d\eta\Big)\\
&\lesssim_{\phi}\lambda^4\left(\lambda\nu\right)^{4-n},
 \end{split}
\end{equation}
provided that $\lambda\nu\ge1$.
Now, given $\varepsilon>0$, and $\nu>0$, we fix
\begin{equation}\label{IP-ChoiceOfLambda}
\lambda=\lambda_{\nu,\varepsilon}=\left(\varepsilon^2\nu^{n-4}\right)^{-\frac{1}{n-8}}
\end{equation}
such that
$\lambda^4\left(\lambda\nu\right)^{4-n}=\varepsilon^2$, and $\lambda\nu=\left(\varepsilon\nu^2\right)^{-\frac{2}{n-8}}>1$.
Independently, by \eqref{IP-DefOfW0}, we see that
\begin{equation*}\label{IP-EstimOfW0}
 \Vert w^{0}(t)\Vert_{\dot{H}^2}\gtrsim_{\phi}t^2+O(t),
\end{equation*}
and, consequently, using \eqref{IP-SmallDispQuantitativeEst}, we get that for $\vert s\vert\le c\vert\log\nu\vert^c$ sufficiently large independently of $\nu$, there holds that
\begin{equation}\label{IP-EstimOfW}
 \Vert w^{\nu}(s)\Vert_{\dot{H}^2}\gtrsim_{\phi}s^2.
\end{equation}
Consequently, using \eqref{IP-DefOfSol}, \eqref{IP-ChoiceOfLambda} and \eqref{IP-EstimOfW}
we get that
\begin{equation}\label{IP-EstimOfU}
\begin{split}
\Vert u^{(\nu,\lambda)}(\lambda^{-4}t)\Vert_{H^2}^2
&\ge\Vert u^{(\nu,\lambda)}(\lambda^{-4}t)\Vert_{\dot{H}^2}^2\\
&\ge \lambda^4\left(\lambda\nu\right)^{4-n}\Vert w^\nu(t)\Vert_{\dot{H}^2}^2\\
&\gtrsim_{\phi}\varepsilon^2 t^4
\end{split}
\end{equation}
for $t$ sufficiently large.
Now, given $\varepsilon$, we let $\nu>0$ be sufficiently small such that
\begin{equation}\label{IP-ChoiceOfNu}
\begin{split}
&\varepsilon^2 t_\nu^4>\varepsilon^{-2}\hskip.1cm,\hskip.1cm\hbox{for}\hskip.1cm t_\nu=c\vert\log\nu\vert^c\hskip.1cm,\hskip.1cm\hbox{and}\\ &\varepsilon^\frac{16-n}{n-8}\nu^\frac{4(n-4)}{n-8}<\varepsilon.
\end{split}
\end{equation}
We choose $\lambda=\lambda_{\nu,\varepsilon}$ as in \eqref{IP-ChoiceOfLambda}. Using \eqref{IP-ChoiceOfNu}, we get that $t_\varepsilon=\lambda^{-4}t_\nu<\varepsilon$, and then \eqref{IP-CompOFHSNorm} and \eqref{IP-EstimOfU} give \eqref{IP-IllPosednessResult}.
This finishes the proof.
\end{proof}


\section{Reduction to three scenarii}\label{Section-3Scenarii}

>From now on we start with the analysis of the energy-critical case $n=8$. In this section we prove that the analysis can be reduced to the study of some very special solutions. In order to do so, we borrow ideas from previous works developed in the context of Schr\"odinger and wave equations by Bahouri and Gerard \cite{BahGer}, Kenig and Merle \cite{KenMer}, Keraani \cite{Ker},  Killip, Tao and Visan \cite{KilTaoVis}, and Tao, Visan and Zhang \cite{TaoVisZha}. We refer also to Pausader \cite{Pau2} for a similar result developed in the context of the $L^2$-critical fourth-order Schr\"odinger equation.
For any $E>0$, we let
\begin{equation}\label{3S-DefOfS}
 \Lambda(E)=\sup\{\Vert u\Vert_{Z(I)}^6: E(u)\le E\},
\end{equation}
where the supremum is taken over all maximal-lifespan solutions $u\in C(I,\dot{H}^2)$ of \eqref{CubicFOS} satisfying $E(u)\le E$. In light of Proposition \ref{Not-BCC-ClaimOfLE} and of the Strichartz estimates \eqref{Not-SE-StrichartzEstimatesWithGainI}, we know that there exists $\delta>0$ such that, for any $E\le \delta$, $\Lambda(E)\lesssim_\delta E<+\infty$. Besides, $\Lambda$ is clearly an increasing function of $E$. Hence, we can define \begin{equation}\label{3S-DefOfEmax}
E_{max}=\sup\{E>0:\Lambda(E)<\infty\}.
\end{equation}
The goal in Sections \ref{Section-3Scenarii}--\ref{Section-Cascade} is to prove that $E_{max}=+\infty$. Theorem \ref{3S-H2CritThm3S} below is a first step in this direction.

\begin{theorem}\label{3S-H2CritThm3S}
 Suppose that $E_{max}<+\infty$. There exists $u\in C(I,\dot{H}^2)$ a maximal-lifespan solution of energy exactly $E_{max}$ such that the $Z(I^\prime)$-norm of $u$ is infinite for $I^\prime = (T_\ast,0)$ and $I^\prime = (0,T^\ast)$, where $I=(T_\ast,T^\ast)$. Besides, 
there exist two smooth functions $h: I \to \mathbb{R}_+^\ast$ and $x: I \to \mathbb{R}^n$ such that 
\begin{equation}\label{3S-DefOk}
K=\{g_{(h(t),x(t))}u(t):t\in I\}
\end{equation}
is precompact in $\dot{H}^2$, where the transformation $g(t)=g_{(h(t),x(t))}$ is as in \eqref{I-DefOfG}. Furthermore, one can assume that 
one of the following three scenarii holds true:
(soliton-like solution) there holds $I=\mathbb{R}$ and $h(t)=1$ for all $t$;  
(double low-to-high cascade) there holds 
$\liminf_{t\to \bar{T}}h(t)=0$ for $\bar{T}=T_\ast,T^\ast$, and $h(t)\le 1$ for all $t$;
(self-similar solution) there holds $I=(0,+\infty)$ and $h(t)=t^{\frac{1}{4}}$ for all $t$.
\end{theorem}

As a remark, since $E(u)=E_{max}$, the solution $u$ in Theorem \ref{3S-H2CritThm3S} is such that $u\ne 0$. Assuming Propositions \ref{SS-NoSelfSimilarSolutionProp}, \ref{Sol-NoSolitonProp} and \ref{Cas-L2Regprop} which exclude the three scenarii in Theorem \ref{3S-H2CritThm3S}, the following corollary holds true.

\begin{corollary}\label{3S-Cor}
For any $E>0$, there exists $C=C(E)$ such that, for any $u_0\in \dot{H}^2$ satisfying $E(u_0)\le E$, if $u\in C(I,\dot{H}^2)$ is the maximal solution of \eqref{CubicFOS} with initial data $u(0)=u_0$, then $I=\mathbb{R}$ and $\Vert u\Vert_{\dot{S}^2(\mathbb{R})}\le C$.
\end{corollary}

\begin{proof}[Proof of Corollary \ref{3S-Cor}]

First, using \cite[Proposition $2.6.$]{Pau1}, we see that a bound on the $Z$-norm of $u$ implies a bound on the $\dot{S}^2$-norm of $u$. Hence if Corollary \ref{3S-Cor} is false, then $E_{max}<+\infty$. Applying Theorem \ref{3S-H2CritThm3S}, we find a
maximal solution satisfying one of the three scenarii in Theorem \ref{3S-H2CritThm3S}. Then, using Propositions \ref{SS-NoSelfSimilarSolutionProp}, \ref{Sol-NoSolitonProp} and \ref{Cas-L2Regprop}, we get a contradiction. Hence $E_{max}=+\infty$.
\end{proof}

Now we prove Theorem \ref{3S-H2CritThm3S}.

\begin{proof}[Proof of Theorem \ref{3S-H2CritThm3S}]
In several ways the proof is similar to the one developed in the $L^2$-critical case in Pausader \cite{Pau2}. We prove the more general statement that Theorem \ref{3S-H2CritThm3S} holds true in any dimension $n\ge 5$ when \eqref{CubicFOS} is replaced by the $\dot{H}^2$-critical equation. In particular, this is the case when $n=8$. Therefore, in this proof, \eqref{CubicFOS} always refers to the energy-critical equation in dimension $n$, and the energy $E$ and $\Lambda$ must be replaced by
\begin{equation*}
 \begin{split}
  &E(u)=\int_{\mathbb{R}^n}\left(\frac{1}{2}\vert\Delta u(x)\vert^2+\frac{n-4}{2n}\vert u(x)\vert^\frac{2n}{n-4}\right)dx\hskip.2cm\hbox{and}\\
  &\Lambda(E)=\sup\{\Vert u\Vert_Z^\frac{2(n+4)}{n-4}:E(u)\le E\},
 \end{split}
\end{equation*}
where the supremum is taken over all maximal solutions of the energy-critical equation of energy less or equal to $E$. Besides, the definition of $\tau$ and $g$ as in \eqref{I-RescalingLaw} and \eqref{I-DefOfG} and Propositions \ref{Not-BCC-ClaimOfLE} and \ref{Not-S-StabProp2} refer to their $n$-dimensional energy-critical counterparts.
A consequence of the precised Sobolev's inequality in Gerard, Meyer and Oru \cite{GerMeyOru} and of the Strichartz estimates \eqref{Not-SE-StrichartzEstimatesWithGainI} is that, for any $u_0\in\dot{H}^2$,
\begin{equation}\label{3S-PrecisedSobolevInequality}
\begin{split}
\Vert e^{it\Delta^2}u_0\Vert_{Z(\mathbb{R})}&\lesssim \Vert e^{it\Delta^2}\vert\nabla\vert u_0\Vert_{L^\frac{2(n+4)}{n-2}L^\frac{2(n+4)}{n-2}}^\frac{n-4}{n-2}\Vert e^{it\Delta^2}\vert\nabla\vert u_0\Vert_{L^\infty L^\frac{2n}{n-2}}^\frac{2}{n-2}\\
&\lesssim \Vert u_0\Vert_{\dot{H}^2}^\frac{n-4}{n-2}\Vert e^{it\Delta^2}\vert \nabla\vert u_0\Vert_{L^\infty\dot{H}^1}^\frac{2}{n}\Vert e^{it\Delta^2}\vert\nabla\vert u_0\Vert_{L^\infty \dot{B}^1_{2,\infty}}^\frac{4}{n(n-2)}\\
&\lesssim \Vert u_0\Vert_{\dot{H}^2}^\frac{n^2-2n-4}{n(n-2)}\Vert u_0\Vert_{\dot{B}^2_{2,\infty}}^\frac{4}{n(n-2)},
\end{split}
\end{equation}
where for $s=1,2$, $\dot{B}^s_{2,\infty}$ is a standard homogeneous Besov space. Now, thanks to \eqref{3S-PrecisedSobolevInequality},
we may follow the analysis in Bahouri and Gerard \cite{BahGer} and Keraani \cite{Ker}. In the following, we call scale-core a sequence $(h_k,t_k,x_k)$ such that for every $k$, $h_k>0$, $t_k\in\mathbb{R}$ and $x_k\in\mathbb{R}^n$. Mimicking the proof in Keraani \cite{Ker}
we obtain that
for $(v_k)_k$ a bounded sequence in $\dot{H}^2$,
there exists a sequence $(V^\alpha)_\alpha$ in $\dot{H}^2$, and scale-cores $(h^\alpha_k,t^\alpha_k,x^\alpha_k)$ such that for any $\alpha\ne\beta$,
\begin{equation}\label{3S-OrthogonalSequences}
\left\vert\log\frac{h_k^\alpha}{h_k^\beta}\right\vert+\left(h_k^\alpha\right)^4\left\vert t_k^\alpha-t_k^\beta\right\vert+h_k^\alpha\left\vert x_k^\alpha-x_k^\beta\right\vert\to+\infty
\end{equation}
as $k\to +\infty$,
with the property that, up to a subsequence, for any $A \ge 1$,
\begin{equation}\label{3S-ZZZ} 
v_k= \sum_ {\alpha=1}^A g_{(h_k^\alpha,x_k^\alpha)}\left(e^{-i\left(h_k^\alpha\right)^4t_k^\alpha\Delta^2}V^\alpha\right)+w_k^A
\end{equation}
for all $k$, where $w_k^A \in \dot{H}^2$ for all $k$ and $A$, and
\begin{equation}\label{3S-Doublestar}
\lim_{A\to +\infty}\limsup_{k\to +\infty}\Vert e^{it\Delta^2}w_k^A\Vert_Z = 0.
\end{equation} 
Moreover, we have the following estimates:
\begin{equation}\label{3S-TTT}
\begin{split}
&\Vert e^{it\Delta^2}v_k\Vert_Z^{\frac{2(n+4)}{n-4}} = \sum_{\alpha=1}^{+\infty}\Vert e^{it\Delta^2}V^\alpha\Vert_Z^{\frac{2(n+4)}{n-4}} + o(1)\hskip.1cm\hbox{and},\\
&E(v_k)=\sum_{\alpha=1}^AE(e^{-i(h_k^\alpha)^4t_k\Delta^2}V^\alpha)+\Vert w_k^A\Vert_{\dot{H}^2}^2+o(1)
\end{split}
\end{equation}
for all $k$, where $o(1) \to 0$ as $k \to +\infty$.
Let $(V,(h_k)_k,(t_k)_k,(x_k)_k)$ be such that $V\in \dot{H}^2$, and $(h_k,t_k,x_k)\in\mathbb{R}_+\times\mathbb{R}\times\mathbb{R}^n$ is a scale-core such that $h_k^4t_k$ has a limit $l\in[-\infty,+\infty]$ as $k\to +\infty$. We say that $U$ is the nonlinear profile associated to $(V,(h_k)_k,(t_k)_k,(x_k)_k)$ if $U$ is a solution of \eqref{CubicFOS} defined on a neighborhood of $-l$, and
\begin{equation*}
\Vert U(-h_k^4t_k)-e^{-ih_k^4t_k\Delta^2}V\Vert_{\dot{H}^2}\to 0
\end{equation*}
as $k\to +\infty$. Using the analysis in Pausader \cite{Pau1}, it is easily seen that a nonlinear profile always exists and is unique.
Besides if
\begin{equation}\label{3S-AAA}
E(U)=\lim_kE(e^{-ih_k^4t_k\Delta^2}V)
\end{equation}
is such that $E(U)<E_{max}$, then the associated nonlinear profile $U$ is globally defined, and
\begin{equation*}\label{3S-LastEqt2}
\Vert U\Vert_{\dot{S}^2(\mathbb{R})}\lesssim_{E(U)} 1.
\end{equation*}
Now, we enter more specifically into the proof of Theorem \ref{3S-H2CritThm3S}. A consequence of Proposition \ref{Not-S-StabProp2} is that there exists a sequence of nonlinear solutions $u_k$ such that $E(u_k)<E_{max}$, $E(u_k)\to E_{max}$, and
\begin{equation}\label{3S-Blow-upCrit}
\Vert u_k\Vert_{Z(-\infty,0)}\hskip.1cm,\hskip.1cm \Vert u_k\Vert_{Z(0,+\infty)}\to +\infty.
\end{equation}
We let $((h^\alpha_k)_k,(t^\alpha_k)_k,(x^\alpha_k)_k)=({\bf h}^\alpha,{\bf z}^\alpha)$, $V^\alpha$, and ${\bf w}^A$ be given by \eqref{3S-ZZZ} applied to the sequence $(v_k=u_k(0))_k$.
Passing to subsequences, and using a diagonal extraction argument, we can assume that, for all $\alpha$, $\left(h_k^\alpha\right)^4t_k^\alpha$ has a limit in $[-\infty,\infty]$. We let $U^\alpha$ be the nonlinear profile associated to $(V^\alpha,{\bf h}^\alpha,{\bf z}^\alpha)$.
Suppose first that there exists $\alpha$ such that $0<E(U^\alpha)<E_{max}$. Then, applying \eqref{3S-TTT} and \eqref{3S-AAA}, we see that there exists $\varepsilon>0$ such that for any $\beta$, $E\left(U^\beta\right)<E_{max}-\varepsilon$, and we get that all the nonlinear profiles are globally defined. Letting $W_k^A(t)=e^{it\Delta^2}w_k^A$, we remark that \begin{equation*}
p_k^A=\sum_{\alpha=1}^A\tau_{(h^\alpha_k,z^\alpha_k)}U^\alpha+W_k^A
\end{equation*}
satisfies \eqref{Not-S-AlmostSolution} with \begin{equation*}
e=e_k^A=f(\sum_{\alpha=1}^A\tau_{(h^\alpha_k,z^\alpha_k)}U^\alpha+W_k^A)-\sum_{\alpha=1}^Af(\tau_{(h^\alpha_k,z^\alpha_k)}U^\alpha)
\end{equation*}
and initial data $p_k^A(0)=u_k(0)+o_A(1)$, where $f(x)=\vert x\vert^\frac{8}{n-4}x$.
First, we claim that
\begin{equation}\label{3S-SumNonlinearProfileBoundedX}
\limsup_k\Vert \sum_{\alpha=1}^A\tau_{(h^\alpha_k,z^\alpha_k)}U^\alpha\Vert_{Z}\lesssim_{E_{max},\varepsilon}1
\end{equation}
independently of $A$.
Indeed, we remark that when $(h^\alpha_k,t^\alpha_k,x^\alpha_k)$ and $(h^\beta_k,t^\beta_k,x^\beta_k)$ satisfy \eqref{3S-OrthogonalSequences}, then for any $u$, $v$ with finite $Z$-norm, there holds that
\begin{equation}\label{3S-WeakConvStat}
\Vert \vert \tau_{(h^\beta_k,t^\beta_k,x^\beta_k)} v\vert^\frac{n+12}{n-4}\tau_{(h^\alpha_k,t^\alpha_k,x^\alpha_k)}u\Vert_{L^1(\mathbb{R},L^1)}\to 0
\end{equation}
as $k\to +\infty$, where $\tau_{(h_k,t_k,x_k)}$ is as in \eqref{I-RescalingLaw}.
Now, since $\Lambda$ is sublinear around $0$, and bounded on $[0,E_{max}-\varepsilon]$, using \eqref{3S-TTT} and \eqref{3S-WeakConvStat}, we get that
\begin{equation*}
\begin{split}
\limsup_k\Vert \sum_{\alpha=1}^A\tau_{(h^\alpha_k,z^\alpha_k)}U^\alpha\Vert_{Z}&= \left(\sum_{\alpha=1}^A \Vert U^\alpha\Vert_Z^\frac{2(n+4)}{n-4}\right)^\frac{n-4}{2(n+4)}\\
&\lesssim\left(\sum_{\alpha=1}^A\Lambda(E(U^\alpha))\right)^\frac{n-4}{2(n+4)}\\
&\lesssim_{E_{max},\varepsilon}\left(\sum_{\alpha=1}^A E\left(U^\alpha\right)\right)^\frac{n-4}{2(n+4)}\\
&\lesssim_{E_{max},\varepsilon}1.
\end{split}
\end{equation*}
Using again \eqref{3S-WeakConvStat}, we get that
\begin{equation}\label{3S-ProofNLMainTHmEqt1X1}
\begin{split}
 &\Vert f(\sum_{\alpha=1}^A\tau_{(h^\alpha_k,z^\alpha_k)}U^\alpha)-
\sum_{\alpha=1}^Af(\tau_{(h^\alpha_k,z^\alpha_k)}U^\alpha)\Vert_{L^2(\mathbb{R},L^2)}=o_A(1)
\end{split}
\end{equation}
as $k\to +\infty$.
On the other hand, using the blow-up criterion in Pausader \cite[Proposition $2.6.$]{Pau1}, and the bound $\Vert U^\alpha\Vert_{Z}\le \Lambda\left(E(U^\alpha)\right)\le\Lambda\left(E_{max}-\varepsilon\right)$, we get that, for any $\alpha$,
\begin{equation*}
 \Vert U^\alpha \Vert_{M}\lesssim_{E_{max},\varepsilon}1.
\end{equation*}
Using
the Leibnitz and chain rules for fractional derivative in Kato \cite{Kat} and Visan \cite[Appendix A]{Vis}, we obtain that
\begin{equation}\label{3S-ProofNLMainTHmEqt1X2}
 \Vert f(\sum_{\alpha=1}^A\tau_{(h^\alpha_k,z^\alpha_k)}U^\alpha)-
\sum_{\alpha=1}^Af(\tau_{(h^\alpha_k,z^\alpha_k)}U^\alpha)\Vert_{L^2(\mathbb{R},\dot{H}^{\frac{n+8}{n+4},\frac{2n(n+4)}{n^2+6n+16}})}\lesssim_{A,E_{max},\varepsilon} 1.
\end{equation}
Interpolating between \eqref{3S-ProofNLMainTHmEqt1X1} and \eqref{3S-ProofNLMainTHmEqt1X2}, we get that
\begin{equation}\label{3S-ProofNLMainTHmEqt1X}
 \Vert f(\sum_{\alpha=1}^A\tau_{(h^\alpha_k,z^\alpha_k)}U^\alpha)-
\sum_{\alpha=1}^Af(\tau_{(h^\alpha_k,z^\alpha_k)}U^\alpha)\Vert_{N}=o_A(1).
\end{equation}
Now, we claim that, letting $s^A_k=\sum_{\alpha=1}^A\tau_{(h^\alpha_k,z^\alpha_k)}U^\alpha$, there holds that
\begin{equation}\label{3S-SAKBounded}
\limsup_k\Vert s^A_k\Vert_{M}\lesssim_{E_{max},\varepsilon}1,
\end{equation}
independently of $A$. Indeed, $s^A_k$ satisfies the equation
\begin{equation*}\label{3S-EquationForSAk}
i\partial_ts^A_k+\Delta^2s^A_k+\sum_{\alpha=1}^Af(\tau_{(h^\alpha_k,z^\alpha_k)}U^\alpha)=0,
\end{equation*}
with initial data
\begin{equation*}
s^A_k(0)=\sum_{\alpha=1}^A\tau_{(h^\alpha_k,z^\alpha_k)}U^\alpha(0)=\sum_{\alpha=1}^Ag_{(h^\alpha_k,x^\alpha_k)}e^{-i\left(h_k^\alpha\right)^4t^\alpha_k\Delta^2}V^\alpha+o_A(1),
\end{equation*}
and consequently \eqref{3S-TTT} and \eqref{3S-AAA} give that
\begin{equation*}
\Vert s^A_k(0)\Vert_{\dot{H}^2}^2\le 2E\left(s^A_k(0)\right)\lesssim_{E_{max}} 1+o_A(1).
\end{equation*}
Using the Strichartz estimates \eqref{Not-SE-StrichartzEstimatesWithGainI}, \eqref{3S-SumNonlinearProfileBoundedX} and \eqref{3S-ProofNLMainTHmEqt1X}, we get that
\begin{equation}\label{3S-StricEstForSAK}
\begin{split}
\Vert s^A_k\Vert_{M}&\lesssim \Vert s^A_k(0)\Vert_{\dot{H}^2}+\Vert \sum_{\alpha=1}^Af(\tau_{(h^\alpha_k,z^\alpha_k)}U^\alpha)\Vert_{N}\\
&\lesssim E\left(s^A_k(0)\right)^\frac{1}{2}+o_A(1)+\Vert
f(\sum_{\alpha=1}^A\tau_{(h^\alpha_k,z^\alpha_k)}U^\alpha)\Vert_{N}\\
&\lesssim_{E_{max}} 1+o_A(1)+\Vert s^A_k\Vert_Z^\frac{8}{n-4}\Vert s^A_k\Vert_{W}\\
&\lesssim_{E_{max}} 1+o_A(1)+\Vert s^A_k\Vert_{Z}^\frac{8}{n-4}\Vert s^A_k\Vert_Z^\frac{1}{2}\Vert s^A_k\Vert_{M}^\frac{1}{2}\\
&\lesssim_{E_{max},\varepsilon}1+o_A(1)+\Vert s^A_k\Vert_{M}^\frac{1}{2}\\
&\lesssim_{E_{max},\varepsilon}1+o_A(1)
\end{split}
\end{equation}
and \eqref{3S-StricEstForSAK} proves \eqref{3S-SAKBounded}.
Independently,
\begin{equation}\label{3S-ProofNLMainTHmEqt2X1}
 \begin{split}
  &\Vert f(\sum_{\alpha=1}^A\tau_{(h^\alpha_k,z^\alpha_k)}U^\alpha+W_k^A)-f(\sum_{\alpha=1}^A\tau_{(h^\alpha_k,z^\alpha_k)}U^\alpha)\Vert_{L^2(\mathbb{R},L^2)}\\
&\lesssim \Vert W_k^A\Vert_{Z}\left(\Vert W_k^A\Vert_{Z}^\frac{8}{n-4}+\Vert \sum_{\alpha=1}^A\tau_{(h^\alpha_k,z^\alpha_k)}U^\alpha\Vert_{Z}^\frac{8}{n-4}\right)\\
&\lesssim_{E_{max},\varepsilon} \Vert W_k^A\Vert_{Z}\left(\Vert W_k^A\Vert_{Z}^\frac{8}{n-4}+1\right)\\
&\lesssim_{E_{max},\varepsilon} \Vert W_k^A\Vert_{Z}
 \end{split}
\end{equation}
and again, using \eqref{3S-SAKBounded} and the product and Leibnitz rules for fractional derivatives, we get that
\begin{equation}\label{3S-ProofNLMainTHmEqt2X2}
 \Vert f(\sum_{\alpha=1}^A\tau_{(h^\alpha_k,z^\alpha_k)}U^\alpha+W_k^A)-f(\sum_{\alpha=1}^A\tau_{(h^\alpha_k,z^\alpha_k)}U^\alpha)\Vert_{L^2(\mathbb{R},\dot{H}^{\frac{n+8}{n+4},\frac{2n(n+4)}{n^2+6n+1}})}\lesssim_{E_{max},\varepsilon}1.
\end{equation}
Interpolating between \eqref{3S-ProofNLMainTHmEqt2X1} and \eqref{3S-ProofNLMainTHmEqt2X2}, we obtain that
\begin{equation}\label{3S-ProofNLMainTHmEqt2X}
 \Vert f(\sum_{\alpha=1}^A\tau_{(h^\alpha_k,z^\alpha_k)}U^\alpha+W_k^A)-f(\sum_{\alpha=1}^A\tau_{(h^\alpha_k,z^\alpha_k)}U^\alpha)\Vert_{N}\lesssim_{E_{max},\varepsilon}\Vert W_k^A\Vert_{Z}^\frac{4}{n+8}
\end{equation}
and \eqref{3S-Doublestar}, \eqref{3S-ProofNLMainTHmEqt1X} and \eqref{3S-ProofNLMainTHmEqt2X} show that
\begin{equation}\label{3S-ProofAddedNewEst1}
\limsup_k\Vert e^A_k\Vert_{N}=o(1)
\end{equation}
as $A\to +\infty$.
Independently,
\begin{equation}\label{3S-BoundOnpkANormX}
\begin{split}
\Vert p_k^A\Vert_{W}&\le \Vert \sum_{\alpha=1}^A\tau_{(h^\alpha_k,z^\alpha_k)}U^\alpha\Vert_{W}+\Vert W_k^A\Vert_{W}\\
&\lesssim_{E_{max},\varepsilon} 1+o_A(1).
\end{split}
\end{equation}
Now using Proposition \ref{Not-S-StabProp2}, \eqref{3S-ProofAddedNewEst1} and \eqref{3S-BoundOnpkANormX}, since $p_k^A(0)=u_k(0)+o_A(1)$, we get that
\begin{equation*}
\begin{split}
 \limsup_k \Vert u_k\Vert_Z^\frac{2(n+4)}{n-4}&\lesssim \lim_{A\to +\infty}\limsup_k\Vert p_k^A\Vert_Z^\frac{2(n+4)}{n-4}\\
 &\lesssim \sum_\alpha \Vert U^\alpha\Vert_Z^\frac{2(n+4)}{n-4}\lesssim_{E_{max},\varepsilon}\sum_\alpha E\left(U^\alpha\right)\lesssim_{E_{max},\varepsilon}1
 \end{split}
\end{equation*}
and this contradicts \eqref{3S-Blow-upCrit}.
Now, suppose that for all $\alpha$, we have that $V^\alpha=0$. Then Strichartz estimates \eqref{Not-SE-StrichartzEstimatesWithGainI} and \eqref{3S-TTT} give that
\begin{equation*}
\begin{split}
\Vert e^{it\Delta^2}u_k(0)\Vert_W&\le\Vert e^{it\Delta^2}u_k(0)\Vert_M^\frac{1}{2}\Vert e^{it\Delta^2}u_k(0)\Vert_{Z}^\frac{1}{2}\\
&\lesssim E_{max}^\frac{1}{2}\Vert e^{it\Delta^2}u_k(0)\Vert_{Z}^\frac{1}{2}\to 0
\end{split}
\end{equation*} as $k\to +\infty$, and Proposition \ref{Not-BCC-ClaimOfLE} gives that $\Vert u_k\Vert_Z\to 0$, which contradicts \eqref{3S-Blow-upCrit}.
Consequently, we know that there exists a scale core $(h_k,t_k,y_k)$, and $V\in \dot{H}^2$ such that
\begin{equation*}
u_k(0)=g_{(h_k,y_k)}e^{-it_kh_k^4\Delta^2}V+w_k,
\end{equation*}
where $E\left(w_k\right)\to 0$.
Now, up to passing to a subsequence, we can assume that $t_kh_k^4\to l\in [-\infty,+\infty]$. If $l\in\mathbb{R}$, then, 
replacing $V$ by $e^{-il\Delta^2}V$, we can assume that $l=0$, and changing slightly $w_k$, we can assume that for any $k$, $t_k=0$.
We then get that $u_k(0)=g_{(h_k,y_k)}V+o(1)$ in $\dot{H}^2$, and in particular $E(V)=E_{max}$.
Otherwise, by time reversal symmetry, we can assume that $l=-\infty$, and then, we find that
\begin{equation*}\label{Contradl=InftyEqt1}
 \begin{split}
  \Vert e^{it\Delta^2}u_k(0)\Vert_{Z([0,+\infty))}&\le \Vert \tau_{(h_k,t_k,y_k)}\left(e^{it\Delta^2}V\right)\Vert_{Z([0,+\infty))}+\Vert w_k\Vert_{Z([0,+\infty))}\\
&\le \Vert e^{it\Delta^2}V\Vert_{Z([-h_k^4t_k,+\infty))}+o(1)\\
&=o(1),
 \end{split}
\end{equation*}
and by standard developements, we get that, for $k$ sufficiently large, $\Vert u_k\Vert_{Z(\mathbb{R}_+)}$ remains bounded. Once again, this contradicts \eqref{3S-Blow-upCrit}.
Let $U$ be the maximal nonlinear solution of \eqref{CubicFOS} with initial data $V$, defined on $I=(-T_\ast,T^\ast)$. Suppose, for example that $T^\ast=+\infty$, and that $\Vert U\Vert_{Z(\mathbb{R}_+)}<+\infty$. Then, using Proposition \ref{Not-S-StabProp2} on $\mathbb{R}_+$ with $v=U$, and $u=\tau_{(h_k^{-1},0,-y_k)}u_k$,
we see that $\Vert u_k\Vert_{Z(\mathbb{R}_+)}$ is bounded uniformly in $k$, which is a contradiction with \eqref{3S-Blow-upCrit}. Consequently, we have that
\begin{equation*}
\Vert U\Vert_{Z(0,T^\ast)}=\Vert U\Vert_{Z(-T_\ast,0)}=+\infty
\end{equation*}
and $E(U)=E_{max}$.
Now, we prove the compactness property of $U$.
In the sequel, we let $N_{min}>0$ be sufficiently small so that $\Vert u\Vert_{\dot{H}^2}\le N_{min}$ implies  $E(u)<E_{max}/4$.
Proceeding as above, it is easily proved by contradiction that for any $\varepsilon>0$, there exist $t_1,\dots,t_j$, $j=j(\varepsilon)$, such that for any time $t\in (-T_\ast,T^\ast)$, there exist $i=i(t)$, and $g(t)=g_{(h(t),y(t))}$ with the property that
$\Vert u(t_i)-g(t)u(t)\Vert_{\dot{H}^2}\le\varepsilon$.
Let us apply this with $\varepsilon=N_{min}$. We get a function $g(t)=g_{(h(t),y(t))}$, and a finite set of times $t_1,\dots, t_j$ such that for any $t$, there exists $i$ satisfying
\begin{equation*}
\Vert u(t_i)-g(t)u(t)\Vert_{\dot{H}^2}\le N_{min}.
\end{equation*}
We claim that $K=\{g(t)u(t):t\in (-T_\ast,T^\ast)\}$ is precompact in $\dot{H}^2$.
Suppose by contradiction that this is not true. Then, there exist $\varepsilon>0$, and a sequence $s_k$ such that for any $k$ and $p$,
\begin{equation}\label{3S-EqtCompactnessContrad}
\Vert g(s_k)u(s_k)-g(s_{p})u(s_{p})\Vert_{\dot{H}^2}>\varepsilon.
\end{equation}
According to what we said above, and passing to a subsequence, we can assume that there exist two times $\bar{t},\bar{t}^\prime$, and a sequence $g^\prime_k=g_{(h_k^\prime,y_k^\prime)}$ such that, for any $k$,
\begin{equation}\label{3S-EqtCompactness1}
\begin{split}
&\Vert u(\bar{t})-g(s_k)u(s_k)\Vert_{\dot{H}^2}<N_{min},\hskip.1cm\hbox{and}\\
&\Vert u(\bar{t}^\prime)-g^\prime_ku(s_k)\Vert_{\dot{H}^2}<\frac{\varepsilon}{4}.
\end{split}
\end{equation}
Passing to a subsequence, it is easily seen that
that $\left(h_k^\prime\right)^{-1}h(s_k)$ remains in a compact subset of $(0,\infty)$ and that and 
$y(s_k)-h(s_k)^{-1}h_k^\prime y_k^\prime$ remains in a compact subset of $\mathbb{R}^n$. Hence, up to considering a subsequence, we can find $g_\infty$ such that $ g(s_k)\left(g_k^\prime\right)^{-1}\to g_\infty$ strongly.
Now, using \eqref{3S-EqtCompactness1} and the fact that $g_{(h,y)}$ is an isometry on $\dot{H}^2$ for all $(h,y)$, we get that
\begin{equation*}\label{3S-EqtCompactness2}
\begin{split}
&\Vert g(s_k)u(s_k)-g(s_{k+1})u(s_{k+1})\Vert_{\dot{H}^2}\\
&\le \Vert g(s_k)u(s_k)-g_\infty u(\bar{t}^\prime)\Vert_{\dot{H}^2}+\Vert g_\infty u(\bar{t}^\prime)-g(s_{k+1})u(s_{k+1})\Vert_{\dot{H}^2}\\
&\le \Vert g^\prime_ku(s_k)-g^\prime_kg(s_k)^{-1}g_\infty u(\bar{t}^\prime)\Vert_{\dot{H}^2}+\Vert g^\prime_{k+1}u(s_{k+1})-g^\prime_{k+1}g(s_{k+1})^{-1}g_\infty u(\bar{t}^\prime)\Vert_{\dot{H}^2}\\
&\le \frac{\varepsilon}{2}+o(1).
\end{split}
\end{equation*}
Clearly, this contradicts \eqref{3S-EqtCompactnessContrad} and proves the compactness property of $K$.
The remaining part follows the line of the work in Tao, Visan and Zhang \cite{TaoVisZha} and Killip, Tao and Visan \cite{KilTaoVis}.
However, in order to obtain a low-to-high cascade (instead of a high-to-low cascade), we make the following slight modification. We use the notations in Killip, Tao and Visan \cite{KilTaoVis}, except for $h(t)=N(t)^{-1}$. In case $Osc(\kappa)$ is unbounded, instead of $a$, we introduce the quantity
\begin{equation*}\label{DefOfBFunction}
 b(t_0)=\inf\left(\frac{h(t_0)}{\inf_{t\ge t_0}h(t)},\frac{h(t_0)}{\inf_{t\le t_0}h(t)}\right).
\end{equation*}
Then, if $\sup_{t_0\in J}b(t_0)=+\infty$, we can find intervals on which the solution presents arbitrarily large relative peak. In particular it becomes possible to find a solution satisfying the low-to-high cascade scenario.
Finally, in case $\sup_{t_0\in J}b(t_0)<+\infty$, the solution has arbitrarily large oscillation, but no relative peak. Mimicking the proof in Killip, Tao and Visan \cite{KilTaoVis}, but changing future (resp past)-focusing time into future (resp past)-defocusing time, one can find a solution behaving as in the self-similar case scenario. Theorem \ref{3S-H2CritThm3S} follows.
\end{proof}


\section{The self-similar case}\label{Section-Self-Similar}

In this section, we deal with the easiest case in Theorem \ref{3S-H2CritThm3S}, namely, the self-similar-like solution. We prove that it is not consistent with conservation of the energy, compactness up to rescaling, and almost conservation of the local $L^2$-norm as expressed in \eqref{Not-AlmostConservationOfLocalMassEstimate}.
More precisely, we prove that the following proposition holds true.

\begin{proposition}\label{SS-NoSelfSimilarSolutionProp}
Let $u\in C(I,\dot{H}^2)$ be a maximal-lifespan solution such that $K=\{g(t)u(t):t\in I\}$ is precompact in $\dot{H}^2$ for some function $g$ as in \eqref{I-DefOfG}. If $n=8$, and $I\ne\mathbb{R}$, then $u=0$. In particular, the self-similar scenario in Theorem \ref{3S-H2CritThm3S} does not hold true.
\end{proposition}

\begin{proof}
Let $u\in C(I,\dot{H}^2)$ be a solution as above, with $I\ne\mathbb{R}$, and let $v(t)=g(t)u(t)$. Without loss of generality, we can assume that $\inf I=0$ and that $(0,2)\subset I$.
Fix $0<t<1$. First, using H\"older's inequality, we get that, for any $\delta>0$,
\begin{equation}\label{SS-Scenar3ProofEqt1}
\int_{B(-h(t)x(t),\delta)}\vert u(t,x)\vert^2dx\lesssim_{E_{max}}\delta^4.
\end{equation}
Independently, let $x_0\in\mathbb{R}^n$, $R>\delta>0$, $D=B(x_0,R)\setminus B(-h(t)x(t),\delta)$, and $D^\prime=B(x(t)+x_0/h(t),R/h(t))\setminus B(0,\delta/h(t))$. Using H\"older's inequality once again, we get that
\begin{equation}\label{SS-Scenar3ProofEqt2}
\begin{split}
\int_{D}\vert u(t,x)\vert^2dx&=h(t)^4\int_{D^\prime}\vert v(t,x)\vert^2dx\\
&\le h(t)^4\left(\int_{\vert x\vert\ge\frac{\delta}{h(t)}}\vert v(t,x)\vert^{4}dx\right)^\frac{1}{2}\left(\int_{B(\frac{x_0}{h(t)}+x(t),\frac{R}{h(t)})}dx\right)^\frac{1}{2}\\
&\lesssim \epsilon(\delta/h(t))^\frac{1}{2}R^4,
\end{split}
\end{equation}
where $\epsilon$ is given by
\begin{equation*}\label{SS-DefEpsilonR}
\epsilon(R)=\sup_{t\in I}\int_{\vert x\vert\ge R}\vert v(t,x)\vert^{4}dx.
\end{equation*}
A consequence of the compactness of $K$ as in Theorem \ref{3S-H2CritThm3S} is that \begin{equation}\label{SS-EpsilonDecays}
\epsilon(R)\to 0\hskip.1cm,\hskip.1cm\hbox{as}\hskip.1cm R\to +\infty.
\end{equation}
Combining \eqref{SS-Scenar3ProofEqt1} and \eqref{SS-Scenar3ProofEqt2}, we get that for any ball $B_R$ of radius $R>\delta$,
\begin{equation}\label{SS-Scenar3ProofEqt3}
\int_{B_R}\vert u(t,x)\vert^2dx\lesssim_{E_{max}} \delta^4+R^4\epsilon(\delta/h(t))^\frac{1}{2}.
\end{equation}
Using almost conservation of local mass, as expressed in \eqref{Not-AlmostConservationOfLocalMassEstimate}, and \eqref{SS-Scenar3ProofEqt3}, we get, for any $x_0\in\mathbb{R}^8$ and any $R>4$, that the following bound at time $1$ holds true
\begin{equation}\label{SS-Scenar3ProofEqt4}
\begin{split}
M\left(u(1),B(x_0,R)\right)^\frac{3}{4}&\lesssim_{E_{max}} \frac{1}{R}+M\left(u(t),B(x_0,2R)\right)^\frac{3}{4}\\
&\lesssim_{E_{max}} \frac{1}{R}+\left(\delta^4+R^4\epsilon(\delta/h(t))^\frac{1}{2}\right)^\frac{3}{4},
\end{split}
\end{equation}
where the local mass is as in \eqref{Not-DefLocalMass}.
Letting $t\to 0$ and using \eqref{SS-EpsilonDecays}, and then letting $\delta\to 0$, we get with \eqref{SS-Scenar3ProofEqt4} that
\begin{equation}\label{SS-Scenar3ProofEqt5}
M\left(u(1),B(x_0,R)\right)\lesssim_{E_{max}} R^{-\frac{4}{3}}.
\end{equation}
Letting $R\to \infty$ in \eqref{SS-Scenar3ProofEqt5}, we obtain
\begin{equation}\label{SS-Scenar3ProofEqt6}
\Vert u(1)\Vert_{L^2}=0.
\end{equation}
Clearly \eqref{SS-Scenar3ProofEqt6} contradicts $u\ne 0$. This proves Proposition \ref{SS-NoSelfSimilarSolutionProp}.
\end{proof}


\section{An interaction Morawetz estimate}\label{Section-IME}

To deal with the remaining two scenarii in Theorem \ref{3S-H2CritThm3S}, in which there is no prescribed finite-time blow-up, we need a new ingredient that bounds the amount of nonlinear presence of the solution at a given scale. Natural candidates to achieve this are Morawetz estimates and in our case, interaction Morawetz estimates. In light of Theorem \ref{3S-H2CritThm3S}, we need to work exclusively with $\dot{H}^2$-solutions. Interaction Morawetz estimates scale like the $\dot{H}^\frac{1}{4}$-norm.
Because of this $7/4$-differrence in scaling, following Colliander, Keel, Staffilani, Takaoka and Tao \cite{ColKeeStaTakTao}, Ryckman and Visan \cite{RycVis} and Visan \cite{Vis}, we seek for frequency-localized interaction Morawetz estimates. This is the purpose of Sections \ref{Section-IME} and \ref{Section-FLIME}. In Section \ref{Section-IME} we derive an a priori interaction estimate that applies to all solutions $u\in C(H^2)$, and in Section \ref{Section-FLIME} we use it to obtain a frequency-localized version of these estimates. The frequency localized version applies only to the special $\dot{H}^2$-solutions given by Theorem \ref{3S-H2CritThm3S}. We prove here that the following proposition holds true.

\begin{proposition}\label{FLIME-InteractionMor1}
Let $n\ge 7$ and let $u\in C([T_1,T_2],H^2)$ be a solution of \eqref{Not-LinSol}, with forcing term $h\in \dot{\bar{S}}^2([T_1,T_2])+\dot{\bar{S}}^0([T_1,T_2])$. Then the following estimate holds true:
\begin{equation}\label{FLIME-IntMorEst1}
\begin{split}
&\sum_{j=1}^n\int_{T_1}^{T_2}\int_{\mathbb{R}^{2n}}\{h,u\}_m(t,y)\frac{(x-y)_j}{\vert x-y\vert}\{\partial_ju,u\}_m(t,x) dxdydt\\
&+\sum_{j=1}^n\int_{T_1}^{T_2}\int_{\mathbb{R}^{2n}}\vert u(t,y)\vert^2\frac{\left(x-y\right)_j}{\vert x-y\vert}\{h,u\}_p^j(t,x)dxdydt\\
&+\int_{T_1}^{T_2}\int_{\mathbb{R}^{2n}}\frac{\vert u(t,x)\vert^2\vert u(t,y)\vert^2}{\vert x-y\vert^5}dxdydt\lesssim \sup_{t=T_1,T_2}\Vert u(t)\Vert_{L^2}^2\Vert u(t)\Vert_{\dot{H}^\frac{1}{2}}^2,\\
\end{split}
\end{equation}
where $\{,\}_m$ and $\{,\}_p$ are the mass and momentum brackets.
\end{proposition}

In this proposition, the mass and momentum brackets are defined by 
\begin{equation}\label{FLIME-MassBracket}
\begin{split}
\{f,g\}_m=\hbox{Im}(f\bar{g})\hskip.1cm,\hskip.1cm\hbox{and},\{f,g\}_p=\hbox{Re}(f\nabla\bar{g}-g\nabla\bar{f}).\\
\end{split}
\end{equation}
In addition to Proposition \ref{FLIME-InteractionMor1}, in order to exploit the bound given in \eqref{FLIME-IntMorEst1}, we also prove that the following lemma holds true.

\begin{lemma}\label{FLIME-SQRTEstimateLemma}
 Assume $n\ge 6$. Then
\begin{equation}\label{FLIME-SQRTEstimate}
\Vert \vert\nabla\vert^{-\frac{n-5}{4}} u\Vert_{L^4}\simeq \Vert (\sum_NN^{-\frac{n-5}{2}}\vert P_Nu\vert^2)^\frac{1}{2}\Vert_{L^4}\lesssim \Vert \vert\nabla\vert^{-\frac{n-5}{2}}\vert u\vert^2\Vert_{L^2}^\frac{1}{2},
\end{equation}
for all $u\in \dot{H}^2$ such that $\vert\nabla\vert^{-\frac{3}{2}}\vert u\vert^2\in L^2$, where the summation is over all dyadic numbers.
\end{lemma}

\begin{proof}
The equivalence of norms is classical.
 We first claim that for any $g\in \mathcal{S}$, and any $n\ge 6$,
\begin{equation}\label{CorToMorProofEqt1}
\Vert \vert \nabla\vert^{-\frac{n-5}{4}}g\Vert_{L^4}\lesssim \Vert \vert\nabla\vert^{-\frac{n-5}{2}}\vert g\vert^2\Vert_{L^2}^\frac{1}{2}.
\end{equation}
We prove \eqref{CorToMorProofEqt1}.
Let $\phi(\xi)=\vert \xi\vert^{-\frac{n-5}{4}}\left(\psi(\xi)-\psi(2\xi)\right)$ where $\psi$ is as in \eqref{DefLitPalOp}. Using the Cauchy-Schwartz inequality we get that for any dyadic $N$,
\begin{equation}\label{CorToMorProofEqt12}
\begin{split}
&\left(P_N\vert \nabla\vert^{-\frac{n-5}{4}}g\right)(x)\\
&=N^{-\frac{n-5}{4}}\left(g\ast\mathcal{F}^{-1}\left(\phi(\xi/N)\right)\right)(x)\\
&=N^\frac{3n+5}{4}\int_{\mathbb{R}^n}g(x-y)\check{\phi}(Ny)dy\\
&\le  N^\frac{3n+5}{4}\left(\int_{\mathbb{R}^n}\vert g(x-y)\vert^2\vert\check{\phi}(Ny)\vert dy\right)^\frac{1}{2}\left(\int_{\mathbb{R}^n}\vert\check{\phi}(Ny)\vert dy\right)^\frac{1}{2}\\
&\lesssim N^{\frac{n+5}{4}}\left(\int_{\mathbb{R}^n}\vert g(x-y)\vert^2\vert\check{\phi}(Ny)\vert dy\right)^\frac{1}{2}
\end{split}
\end{equation}
uniformly in $N$.
Since $\phi\in\mathcal{S}$, for any $y\in\mathbb{R}^n$, we get
\begin{equation}\label{CorToMorProofEqt13}
\begin{split}
\sum_{N}\left(N\vert y\vert\right)^\frac{n+5}{2}\vert\check{\phi}(Ny)\vert\lesssim \sum_N \left(N\vert y\vert\right)^\frac{n+5}{2}\left(1+N\vert y\vert\right)^{-2n}\lesssim 1,
\end{split}
\end{equation}
where the summation is over all dyadic numbers $N$.
Consequently, using \eqref{CorToMorProofEqt12}, \eqref{CorToMorProofEqt13} and the fact that $\check{\phi}\in\mathcal{S}$, we get that
\begin{equation}\label{CorToMorProofEqt14}
\begin{split}
\sum_N\vert P_N\vert \nabla\vert^{-\frac{n-5}{4}}g\vert^2(x)&\lesssim\sum_N N^\frac{n+5}{2}\int_{\mathbb{R}^n}\vert g(x-y)\vert^2\vert \check{\phi}(Ny)\vert dy\\
&\lesssim\int_{\mathbb{R}^n}\frac{\vert g(x-y)\vert^2}{\vert y\vert^\frac{n+5}{2}}\left(\sum_N\left(N\vert y\vert\right)^\frac{n+5}{2}\vert\check{\phi}(Ny)\vert\right)dy\\
&\lesssim \left(\vert\nabla\vert^{-\frac{n-5}{2}}\vert g\vert^2\right)(x),
\end{split}
\end{equation}
and using the Littlewood-Paley Theorem, \eqref{CorToMorProofEqt14} gives \eqref{CorToMorProofEqt1} for $g$ smooth. Density arguments then give \eqref{FLIME-SQRTEstimate}. This ends the proof of Lemma \ref{FLIME-SQRTEstimateLemma}.
\end{proof}

\begin{proof}[Proof of Proposition \ref{FLIME-InteractionMor1}]
Since the estimate we want to prove is linear, we can assume that $u$ is smooth and use density arguments to recover the general case.
We adopt the convention that repeated indices are summed.
Given some real function $a$, we define the Morawetz action centered at $0$ by
\begin{equation}\label{CL-M-MorawetzAction0}
M_a^0(t)=2\int_{\mathbb{R}^n}\partial_ja(x)\hbox{Im}(\bar{u}(t,x)\partial_ju(t,x))dx.
\end{equation}
Following the computation in Pausader \cite{Pau1}, we get that
\begin{equation}\label{CL-M-TimeDerivativeOFMorawetzAction0aQuelconque}
\begin{split}
\partial_tM_a^0(t)=&2\int_{\mathbb{R}^n}\Big(2\partial_ju\partial_k\bar{u}\partial_{jk}\Delta
a-\frac{1}{2}\left(\Delta^3a\right)\vert
u\vert^2-4\partial_{jk}a\partial_{ik}u\partial_{ij}\bar{u}\\
&+\Delta^2a\vert\nabla u\vert^2+\partial_ja\{u,h\}_p^j\Big)dx.
\end{split}
\end{equation}
Similarly, we define the Morawetz action centered at $y$, $M^y_a(t)=M_{a_y}^0(t)$ for $a_y(x)=\vert x-y\vert$. Finally, we define the interaction Morawetz action by the following formula:
\begin{equation}\label{CL-M-InteractionMorawetzDef}
\begin{split}
M^i(t)&=\int_{\mathbb{R}^n}\vert u(t,y)\vert^2M_a^y(t)dy\\
&=2\hbox{Im}\left(\int_{\mathbb{R}^n}\int_{\mathbb{R}^n}\vert u(t,y)\vert^2\frac{x-y}{\vert x-y\vert}\nabla u(t,x)\bar{u}(t,x)dxdy\right).
\end{split}
\end{equation}
We can directly estimate
\begin{equation}\label{CL-M-EstimationOfIntercationMorawetzNorm}
\vert M^i(t)\vert\le \Vert u\Vert_{L^\infty L^2}^2\Vert u\Vert_{L^\infty\dot{H}^\frac{1}{2}}^2.
\end{equation}
Now, we get an estimate on the variation of $M^i$ by writing that
\begin{equation}\label{CL-M-DerivativeOfInteractionMorawetz}
\begin{split}
\partial_tM^i= &2\int_{\mathbb{R}^n}\{u,h\}_m(y)M^y_ady+4\hbox{Im}\int_{\mathbb{R}^n}\partial_ju(y)\partial_{jk}\bar{u}(y)\partial_kM^y_ady\\
&+2\hbox{Im}\left(\int_{\mathbb{R}^n}\bar{u}(y)\nabla u(y)\nabla\Delta M_a^ydy\right)+\int_{\mathbb{R}^n}\vert u(y)\vert^2\partial_tM^y_ady.
\end{split}
\end{equation}
This gives that
\begin{equation}\label{IntMorEqt1}
 \begin{split}
  &\partial_tM^i=\\
  &4\int_{\mathbb{R}^n\times\mathbb{R}^n}\hbox{Im}\left(\bar{u}(y)\partial_{j}u(y)\right)
\partial_{j}^y\Delta\left(\partial_{k}^xa (x-y)\right)\hbox{Im}\left(\partial_ku(x)\bar{u}(x)\right)dxdy\\
&+8\int_{\mathbb{R}^n\times\mathbb{R}^n}\hbox{Im}\left(\partial_iu(y)\partial_{ij}\bar{u}(y)\right)\partial_{j}^y\left(\partial_{k}^xa(x-y)\right)\hbox{Im}\left(\partial_ku(x)\bar{u}(x)\right)dxdy\\
&+4\int_{\mathbb{R}^n\times\mathbb{R}^n}\{u,h\}_m(y)\partial_{k}^xa(x-y)\hbox{Im}\left(\partial_ku(x)\bar{u}(x)\right)dxdy\\
&+4\int_{\mathbb{R}^n\times\mathbb{R}^n}\vert u(y)\vert^2\partial_{jk}^x\left(\Delta a(x-y)\right)\partial_ju(x)\partial_k\bar{u}(x)dxdy\\
&-\int_{\mathbb{R}^n\times\mathbb{R}^n}\vert u(y)\vert^2\left(\Delta^3a(x-y)\right)\vert u(x)\vert^2dxdy\\
&-8\int_{\mathbb{R}^n\times\mathbb{R}^n}\vert u(y)\vert^2\left(\partial_{jk}^xa(x-y)\right)\partial_{ik}u(x)\partial_{ij}\bar{u}(x)dxdy\\
&+2\int_{\mathbb{R}^n\times\mathbb{R}^n}\vert u(y)\vert^2\left(\Delta^2a(x-y)\right)\vert\nabla u(x)\vert^2dxdy\\
&+2\int_{\mathbb{R}^n\times\mathbb{R}^n}\vert u(y)\vert^2\partial_{j}^xa(x-y)\{u,h\}_p^j(x)dxdy,
\end{split}
\end{equation}
where $\partial_j^x$ denotes derivation with respect to $x_j$, and $\partial_k^y$ derivation with respect to $y_k$.
Most of the terms in \eqref{IntMorEqt1} have the right sign if we let $a(z)=\vert z\vert$. Now we focus on the first two terms in \eqref{IntMorEqt1}.
In the sequel, we let $z=x-y$.
Using the fact that $\hbox{Re}\left(AB\right)=\hbox{Re}\left(A\right)\hbox{Re}\left(B\right)-\hbox{Im}\left(A\right)\hbox{Im}\left(B\right)$, we get the equality:
\begin{equation}\label{BadTerm1}
\begin{split}
&\int_{\mathbb{R}^{2n}}\hbox{Im}\left(\bar{u}(y)\partial_ju(y)\right)\left(\partial_j^y\partial_k^x\Delta a(z)\right)\hbox{Im}\left(\partial_ku(x)\bar{u}(x)\right)dxdy\\
&=-\frac{1}{4}\int_{\mathbb{R}^{2n}}\vert u(y)\vert^2\Delta^3 a(z)\vert u(x)\vert^2dxdy-R((\nabla u\otimes u);(\nabla u\otimes u)),
\end{split}
\end{equation}
where we let $R$ be the bilinear form on $\mathcal{S}(\mathbb{R}^{n},\mathbb{R}^{n})\otimes \mathcal{S}(\mathbb{R}^n,\mathbb{R})$ defined by
\begin{equation}\label{DefOfR}
R((\vec{\alpha}\otimes\beta);(\vec{\gamma}\otimes\delta))=\hbox{Re}\int_{\mathbb{R}^{2n}}\alpha_j(y)\bar{\delta}(y)\left(\partial_{jk}^x\Delta a(z)\right)\bar{\gamma}_k(x)\beta(x)dxdy.
\end{equation}
For the second term, we proceed as follows:
\begin{equation}\label{WWW}
\begin{split}
&\int_{\mathbb{R}^{2n}}\hbox{Im}\left(\partial_{ij}\bar{u}(x)\partial_iu(x)\right)\left(\partial^x_{j}\partial^y_{k}a(z)\right)\hbox{Im}\left(\partial_ku(y)\bar{u}(y)\right)dxdy\\
&=\frac{1}{4}\int_{\mathbb{R}^n}\vert \nabla u(x)\vert^2\Delta^2a(z)\vert u(y)\vert^2dxdy+Q((\nabla\partial_iu\otimes u);(\nabla u\otimes\partial_iu)),
\end{split}
\end{equation}
where we define the quadratic form $Q$ on $\mathcal{S}(\mathbb{R}^n,\mathbb{R}^n)\otimes \mathcal{S}(\mathbb{R}^n,\mathbb{R})$ by
\begin{equation}\label{QuadraticForms}
\begin{split}
Q\left((\vec{\alpha}\otimes \beta);(\vec{\gamma}\otimes\delta)\right)&=\hbox{Re}\int_{\mathbb{R}^{2n}}\alpha_k(x)\bar{\delta}(x)\frac{1}{\vert z\vert}\left(\delta_{jk}-\frac{z_jz_k}{\vert z\vert^2}\right)\bar{\gamma}_j(y)\beta(y)dydx.\\
\end{split}
\end{equation}
As one can check by computing the Fourier transform of its kernel, $Q$ is nonnegative.
Hence, applying the Cauchy-Schwartz inequality, we get
\begin{equation}\label{CauchySchwartz}
\begin{split}
\vert Q((\nabla\partial_iu\otimes u);(\nabla u\otimes \partial_iu))\vert&\le \vert Q((\nabla \partial_iu\otimes u)^2)\vert^\frac{1}{2}\vert Q((\nabla u\otimes \partial_iu)^2)\vert^\frac{1}{2}\\
&\le \frac{1}{2}Q((\nabla \partial_iu\otimes u)^2)+ \frac{1}{2}Q((\nabla u\otimes\partial_iu)^2)
\end{split}
\end{equation}
and if $R$ and $Q$ are as in \eqref{DefOfR} and \eqref{QuadraticForms}, we observe that
\begin{equation}\label{EqtAfterCS}
\begin{split}
Q((\nabla u\otimes \partial_iu)^2)
&=Q((\nabla\partial_iu\otimes u)^2)-R\left((\nabla u\otimes u)^2\right)\\
&+2\hbox{Re}\int_{\mathbb{R}^{2n}}\partial_ku(x)\bar{u}(x)\left(\partial_{ijk}^x a(z)\right)\partial_{ij}\bar{u}(y)u(y)dxdy\\
&=Q((\nabla\partial_iu\otimes u)^2)+R((\nabla u\otimes u)^2)\\
&+\hbox{Re}\int_{\mathbb{R}^{2n}}\vert u(x)\vert^2\left(\partial_{ij}^x\Delta a(z)\right)\partial_i\bar{u}(y)\partial_ju(y)dxdy.
\end{split}
\end{equation}
Consequently, applying \eqref{BadTerm1}, \eqref{WWW}, \eqref{CauchySchwartz} and \eqref{EqtAfterCS}, we get that
\begin{equation}\label{BadTermsEstim}
\begin{split}
&4\int_{\mathbb{R}^n\times\mathbb{R}^n}\hbox{Im}\left(\bar{u}(y)\partial_{j}u(y)\right)
\partial_{j}^y\Delta\left(\partial_{k}^xa (x-y)\right)\hbox{Im}\left(\partial_ku(x)\bar{u}(x)\right)dxdy\\
&+8\int_{\mathbb{R}^n\times\mathbb{R}^n}\hbox{Im}\left(\partial_iu(y)\partial_{ij}\bar{u}(y)\right)\partial_{j}^y\left(\partial_{k}^xa(x-y)\right)\hbox{Im}\left(\partial_ku(x)\bar{u}(x)\right)dxdy\\
&\le -\int_{\mathbb{R}^{2n}}\vert u(y)\vert^2\left(\Delta^3a(z)\right)\vert u(x)\vert^2dxdy+8Q\left((\nabla\partial_iu\otimes u)^2\right)\\
&+2\int_{\mathbb{R}^{2n}}\vert u(y)\vert^2\left(\Delta^2a(z)\right)\vert\nabla u(x)\vert^2dxdy\\
&+4\hbox{Re}\int_{\mathbb{R}^{2n}}\vert u(x)\vert^2\left(\partial_{ij}^x\Delta a(z)\right)\partial_i\bar{u}(y)\partial_ju(y)dxdy.
\end{split}
\end{equation}
Now, for $e\in\mathbb{R}^n$ a vector, and $u$ a function, we define
\begin{equation*}\label{DefVectorDerivatives}
\begin{split}
\nabla_eu=\left(e\cdot\nabla u\right)\frac{e}{\vert e\vert^2}\hskip.1cm,\hskip.1cm\hbox{and},
\nabla_e^\perp u=\nabla u-\nabla_eu.
\end{split}
\end{equation*}
Then, applying the Cauchy-Schwartz inequality, we get that
\begin{equation}\label{EstimForQ}
\begin{split}
&Q((\nabla\partial_iu,u)^2)\\
&=\int_{\mathbb{R}^{2n}}\partial_{ij}\bar{u}(x)u(x)\frac{1}{\vert x-y\vert}\left(\delta_{jk}-\frac{(x-y)_j(x-y)_k}{\vert x-y\vert^2}\right)\partial_{ik}u(y)\bar{u}(y)dxdy\\
&=\int_{\mathbb{R}^{2n}}\frac{\left[u(x)\nabla_{x-y}^\perp \partial_iu(y)\right]\cdot\left[\nabla_{x-y}^\perp \partial_i\bar{u}(x)\bar{u}(y)\right]}{\vert x-y\vert}dxdy\\
&\le\int_{\mathbb{R}^{2n}}\vert u(x)\vert^2\frac{1}{\vert x-y\vert}\vert \nabla_{x-y}^\perp\partial_i u(y)\vert^2dxdy\\
&\le\int_{\mathbb{R}^{2n}}\vert u(x)\vert^2\frac{1}{\vert x-y\vert}\left(\delta_{jk}-\frac{(x-y)_j(x-y)_j}{\vert x-y\vert^2}\right)\partial_{ik}\bar{u}(y)\partial_{ij}u(y)dxdy.
\end{split}
\end{equation}
Finally, \eqref{IntMorEqt1}, \eqref{BadTermsEstim}, and \eqref{EstimForQ} give
\begin{equation}\label{IntMorEqt2}
\begin{split}
\partial_tM^i\le
&-2\int_{\mathbb{R}^n\times\mathbb{R}^n}\vert u(y)\vert^2\left(\Delta^3a(x-y)\right)\vert u(x)\vert^2dxdy\\
&+4\int_{\mathbb{R}^n\times\mathbb{R}^n}\{u,h\}_m(y)\partial_{k}^xa(x-y)\hbox{Im}\left(\partial_ku(x)\bar{u}(x)\right)dxdy\\
&+2\int_{\mathbb{R}^n\times\mathbb{R}^n}\vert u(y)\vert^2\partial_{j}^xa(x-y)\{u,h\}_p^j(x)dxdy\\
&+8\int_{\mathbb{R}^n\times\mathbb{R}^n}\vert u(y)\vert^2\left(\partial_{jk}^x\Delta a(x-y)\right)\partial_ju(x)\partial_k\bar{u}(x)dxdy\\
&+4\int_{\mathbb{R}^n\times\mathbb{R}^n}\vert u(y)\vert^2\left(\Delta^2a(x-y)\right)\vert\nabla u(x)\vert^2dxdy\hskip.1cm .
 \end{split}
\end{equation}
Let $T1$ and $T2$ be the last two terms in \eqref{IntMorEqt2}. Then
\begin{equation*}\label{LastNonpositive}
\begin{split}
&\frac{1}{4(n-1)}\left(T1 + T2\right)\\
&=-\int_{\mathbb{R}^{2n}}\frac{\vert u(y)\vert^2}{\vert x-y\vert^3}\left((n-1)\delta_{jk}-6\frac{(x-y)_j(x-y)_k}{\vert x-y\vert^2}\right)\partial_ju(x)\partial_k\bar{u}(x)dxdy
\end{split}
\end{equation*}
which is nonpositive when $n\ge 7$. Finally, \eqref{IntMorEqt2} and this remark give \eqref{FLIME-InteractionMor1}.
\end{proof}

\section{A frequency-localized interaction Morawetz estimate}\label{Section-FLIME}

The preceding interaction Morawetz estimate is ill-suited for $\dot{H}^2$-solutions. In order to exploit such an estimate in the context of $\dot{H}^2$-solutions, we need to localize it at high frequencies. The difficulty then is to deal with an inequality that scales like the $\dot{H}^\frac{1}{4}$-norm, while using only bounds that scale like the $\dot{H}^2$-norm. To overcome this difference of $7/4$ derivatives, we split the solution into high and low frequencies and  develop an intricate bootstrap argument to get the inequality. This is made possible because we restrict ourselves to the case of the special solutions obtained in Theorem \ref{3S-H2CritThm3S}.
More precisely, we prove that the following proposition holds true.

\begin{proposition}\label{FLIME-FLIMEProp}
Let $n=8$. Let $u\in C(I,\dot{H}^2)$ be a maximal lifespan-solution of \eqref{CubicFOS} such that $K=\{g(t)u(t):t\in I\}$ is precompact in $\dot{H}^2$ and such that $\forall t\in I,\hskip.1cm h(t)\le h(0)=1$. Then, for any sufficiently small $\varepsilon>0$,
\begin{equation}\label{FLIMEPropEst2}
\begin{split}
&\Vert \vert\nabla\vert^{-\frac{3}{2}}\vert P_{\ge 1}u\vert^2\Vert_{L^2(I,L^2)}\lesssim\varepsilon\hskip.1cm,\\
&\Vert P_{\ge 1}u\Vert_{\dot{S}^{-\frac{3}{2}}(I)}\lesssim\varepsilon\hskip.1cm,\hskip.1cm\hbox{and}\hskip.1cm\Vert P_{\le 1}u\Vert_{\dot{S}^2(I)}\lesssim \varepsilon
\end{split}
\end{equation}
up to replacing $u$ by $g_{(N,0)}u$ for some $N$.
\end{proposition}
 
\begin{proof}
We fix $\varepsilon>0$ sufficiently small to be chosen later on.
We remark that for $N$ a dyadic number and for all time,
\begin{equation}\label{Cas-GainOfDerivative2}
\Vert P_{\le N}g^{-1}(t)\left(g(t)u(t)\right)\Vert_{\dot{H}^2}=\Vert P_{\le Nh(t)}\left(g(t)u(t)\right)\Vert_{\dot{H}^2}.
\end{equation}
Hence, by compactness of $K$, and since $h\le 1$, we have that $\Vert P_{\le N}u\Vert_{L^\infty\dot{H}^2}\to 0$ as $N\to 0$. Let $N$ be such that
\begin{equation*}
\Vert P_{\le \varepsilon^{-4}N}u\Vert_{L^\infty\dot{H}^2}\le \frac{\varepsilon}{2}.
\end{equation*}
Replacing $K$ by $Kg_{(\varepsilon^4 N^{-1},0)}$, and modifying slighly $h$, one can assume that
\begin{equation}\label{FLIME-ProofPreliminaryEst1}
\begin{split}
\Vert P_{\le 1}u\Vert_{L^\infty(I,\dot{H}^2)}&\le \varepsilon\hskip.1cm,\hskip.1cm\hbox{and}\\
\Vert P_{\ge 1}u\Vert_{L^\infty(I,\dot{H}^s)}&\le \Vert P_{1\le\cdot<\varepsilon^{-4}}u\Vert_{L^\infty(I,\dot{H}^2)}+\varepsilon^{4(2-s)}\Vert P_{\ge \varepsilon^{-4}}u\Vert_{L^\infty(I,\dot{H}^2)}\\
&\le \varepsilon,
\end{split}
\end{equation}
for $s\le 7/4$.
We let
\begin{equation}\label{FLIME-Proof-DefOfJC}
J\left(C\right)=\{t\ge 0:\Vert \vert\nabla\vert^{-\frac{3}{2}}\vert P_{\ge 1}u\vert^2\Vert_{L^2([0,t],L^2)}\le C\eta \}.
\end{equation}
The first step in the proof is to obtain good Strichartz controls on the high and low-frequency parts of $u$. In the sequel, we let $u_l=P_{< 1}u$, and $u_h=P_{\ge 1}u$. Besides the summations are always over all dyadic numbers, unless otherwise specified.
We claim that for $J=J(2)$, we have that
\begin{equation}\label{FLIME-Proof-Step1}
\begin{split}
&\Vert \vert\nabla\vert^{-\frac{3}{2}}\vert P_{\ge 1}u\vert^2\Vert_{L^2(J,L^2)}\le 2\eta,\\
&\Vert P_{\le 1}u\Vert_{\dot{S}^2(J)}\lesssim\varepsilon\hskip.1cm,\hskip.1cm\hbox{and}\\
&\Vert P_{\ge 1}u\Vert_{\dot{S}^{-\frac{3}{2}}(J)}\lesssim \eta,
\end{split}
\end{equation}
provided that $\varepsilon>0$ is sufficiently small, and that $\varepsilon<\eta$.
In the following, all space-time norms are taken on the interval $J$.
Applying the Strichartz estimates \eqref{Not-SE-StrichartzEstimatesWithGainII}, we get that
\begin{equation}\label{FLIME-Proof-LowEst}
\begin{split}
\Vert P_{\le 1}u\Vert_{\dot{S}^2}
&\lesssim \Vert P_{\le 1}u(0)\Vert_{\dot{H}^2}+\Vert \vert\nabla\vert P_{\le 1}\left(\vert u_l\vert^2u_l\right)\Vert_{L^2(J,L^\frac{8}{5})}\\
&+\sum_{j=0}^2\Vert \vert\nabla\vert P_{\le 1}\mathcal{O}\left(u_l^ju_h^{3-j}\right)\Vert_{L^2(J,L^\frac{8}{5})}\\
&\lesssim \varepsilon+\Vert u_l\Vert_{\dot{S}^2}^3+\sum_{j=0}^2\Vert \vert\nabla\vert P_{\le 1}\mathcal{O}\left(u_l^ju_h^{3-j}\right)\Vert_{L^2(J,L^\frac{8}{5})}.
\end{split}
\end{equation}
Now, we estimate the terms in the sum.
First, using the Bernstein's properties \eqref{BernSobProp} and \eqref{FLIME-ProofPreliminaryEst1}, we get that
\begin{equation}\label{FLIME-Proof-LowEstEqt1}
\begin{split}
\Vert \vert\nabla\vert P_{\le 1}\mathcal{O}\left(u_l^2u_h\right)\Vert_{L^2 L^\frac{8}{5}}&\lesssim \Vert u_l^2u_h\Vert_{L^2L^\frac{8}{5}}\\
&\lesssim \Vert u_l\Vert_{L^4L^8}\Vert u_l\Vert_{L^4L^8}\Vert u_h\Vert_{L^\infty L^\frac{8}{3}}\\
&\lesssim \varepsilon\Vert u_l\Vert_{\dot{S}^2}^2.
\end{split}
\end{equation}
For the next term, we remark that if $N\ge 4M$ and $N\ge 8$, then the Fourier support of $P_NuP_Mv$ is supported in $\{\vert \xi\vert\ge 2\}$, and $P_{\le 1}\left(P_NuP_Mv\right)=0$.
Using this remark, the Bernstein's properties \eqref{BernSobProp}, \eqref{FLIME-ProofPreliminaryEst1} and the Cauchy-Schwartz inequality, we get
\begin{equation}\label{FLIME-Proof-LowEstEqt2}
\begin{split}
&\Vert \vert \nabla\vert P_{\le 1}\mathcal{O}\left(u_lu_h^2\right)\Vert_{L^2 L^\frac{8}{5}}\\
&\lesssim
\Vert P_{\le 1}\mathcal{O}\left(u_lu_h^2\right)\Vert_{L^2 L^\frac{8}{5}}\\
&\lesssim \sum_{M\le 1,N\le 8}\Vert P_{\le 1}\left(P_Mu P_N\mathcal{O}\left(u_h^2\right)\right)\Vert_{L^2L^\frac{8}{5}}\\
&\lesssim \left(\sum_{M\le 1}\Vert P_Mu\Vert_{L^\infty L^8}\right)\left(\sum_{N\le 8}\Vert P_N\mathcal{O}\left(u_h^2\right)\Vert_{L^2L^2}\right)\\
&\lesssim \left(\sum_{M\le 1} M^{-1}\Vert P_Mu\Vert_{L^\infty L^8}^2\right)^\frac{1}{2}\left(\sum_{N\le 8}N^{-3}\Vert P_N\mathcal{O}\left(u_h^2\right)\Vert_{L^2L^2}^2\right)^\frac{1}{2}\\
&\lesssim \Vert u_l\Vert_{L^\infty \dot{H}^2}\Vert \vert\nabla\vert^{-\frac{3}{2}}\vert u_h\vert^2\Vert_{L^2L^2}\\
&\lesssim \varepsilon\eta,
\end{split}
\end{equation}
where we have used in the last inequalities that since $\vert\nabla\vert^{-\frac{3}{2}}$ has a positive kernel, we have that
$\Vert\vert \nabla\vert^{-\frac{3}{2}}\mathcal{O}(u_h^2)\Vert_{L^2L^2}\le \Vert \vert\nabla\vert^{-\frac{3}{2}}\vert u_h\vert^2\Vert_{L^2L^2}$.
We treat the last term similarly as follows, by writing that
\begin{equation}\label{FLIME-Proof-LowEstEqt3}
\begin{split}
&\Vert \vert\nabla\vert P_{\le 1}\mathcal{O}\left(u_h^3\right)\Vert_{L^2 L^\frac{8}{5}}\\
&\lesssim \Vert P_{\le 1}\mathcal{O}\left(u_h^3\right)\Vert_{L^2L^1}\\
&\lesssim \sum_{1\le N\le 8,M\le 32}\Vert P_{\le 1}\left(P_Nu_hP_M\mathcal{O}\left(u_h^2\right)\right)\Vert_{L^2L^1}\\
&+\sum_{N\ge 8,4N\ge M\ge N/4}\Vert P_{\le 1}\left(P_Nu_hP_M\mathcal{O}\left(u_h^2\right)\right)\Vert_{L^2L^1}\\
&\lesssim \left(\sum_M M^3\Vert P_M u_h\Vert_{L^\infty L^2}^2\right)^\frac{1}{2}\left(\sum_MM^{-3}\Vert P_M\vert u_h\vert^2\Vert_{L^2L^2}^2\right)^\frac{1}{2}\\
&\lesssim \Vert u_h\Vert_{L^\infty H^\frac{7}{4}}\Vert \vert\nabla\vert^{-\frac{3}{2}}\vert u_h\vert^2\Vert_{L^2L^2}\\
&\lesssim \varepsilon\eta.
\end{split}
\end{equation}
Finally, we get with \eqref{FLIME-Proof-LowEst}--\eqref{FLIME-Proof-LowEstEqt3} that
\begin{equation}\label{FLIME-Proof-LowEstEqt4}
\begin{split}
\Vert u_l\Vert_{\dot{S}^2}\lesssim\Vert P_{\le 1}u\Vert_{\dot{S}^2}&\lesssim \varepsilon +\eta\varepsilon+\varepsilon\Vert u_l\Vert_{\dot{S}^2}^2+\Vert u_l\Vert_{\dot{S}^2}^3
\end{split}
\end{equation}
and this proves the second inequality in \eqref{FLIME-Proof-Step1} with $u_l$ instead of $P_{\le 1}u$ if $\varepsilon>0$ is sufficiently small. Using again \eqref{FLIME-Proof-LowEstEqt4}, we get the second inequality in \eqref{FLIME-Proof-Step1}.
Now we turn to the control on $u_h$.
Still using the Strichartz estimates \eqref{Not-SE-StrichartzEstimatesWithGainII} and Sobolev's inequality, we get that
\begin{equation}\label{FLIME-Proof-HighEst1}
\begin{split}
\Vert u_h\Vert_{\dot{S}^{-\frac{3}{2}}}&\lesssim \Vert u_h(0)\Vert_{\dot{H}^{-\frac{3}{2}}}+\sum_{j=0}^3\Vert \vert\nabla\vert^{-\frac{5}{2}}P_{\ge 1}\mathcal{O}\left(u_h^ju_l^{3-j}\right)\Vert_{L^2L^\frac{8}{5}}\\
&\lesssim \varepsilon
+\sum_{j=2,3}\Vert P_{\ge 1}\mathcal{O}\left(u_h^ju_l^{3-j}\right)\Vert_{L^2L^\frac{16}{15}}\\
&+\sum_{j=0,1}\Vert\vert\nabla\vert^{-\frac{5}{2}}P_{\ge 1}\mathcal{O}\left(u_h^ju_l^{3-j}\right)\Vert_{L^2L^\frac{8}{5}}.
\end{split}
\end{equation}
By convolution estimate, letting $c_N=N^{-\frac{3}{4}}\vert P_Nu_h\vert$, we get that
\begin{equation}\label{FLIME-Proof-HighEst2}
\begin{split}
\left\vert\vert u_h\vert^2u_h\right\vert
&\lesssim \vert \sum_{M_1\ge M_2\ge M_3}\mathcal{O}\left(P_{M_1}u_hP_{M_2}u_hP_{M_3}u_h\right)\vert\\
&\lesssim \vert\sum_{M_1\ge M_2\ge M_3}c_{M_3}\left(\frac{M_3}{M_2}\right)^\frac{3}{4}c_{M_2}\left(\frac{M_2}{M_1}\right)^\frac{3}{2}M_1^\frac{3}{2} P_{M_1}u_h\vert\\
&\lesssim \left(\sum_M c_M^2\right)\left(\sup_MM^\frac{3}{2}\vert P_Mu_h\vert\right).
\end{split}
\end{equation}
Consequently, using the Bernstein's properties \eqref{BernSobProp}, \eqref{FLIME-SQRTEstimate} and \eqref{FLIME-Proof-HighEst2},
we get that
\begin{equation}\label{FLIME-Proof-HighEst3}
 \begin{split}
  &\Vert P_{\ge 1}\vert u_h\vert^2u_h\Vert_{L^2L^\frac{16}{15}}\\
  &\lesssim \Vert \vert u_h\vert^2u_h\Vert_{L^2L^\frac{16}{15}}\\
  &\lesssim \Vert (\sum_M M^{-\frac{3}{2}}\vert P_Mu_h\vert^2)^\frac{1}{2}\Vert_{L^4L^4}^2\left(\Vert\sup_M \vert\nabla\vert^\frac{3}{2}P_Mu_h\Vert_{L^\infty L^\frac{16}{7}}\right)\\
  &\lesssim \Vert \vert\nabla\vert^{-\frac{3}{2}}\vert u_h\vert^2\Vert_{L^2L^2}\Vert u_h\Vert_{L^\infty\dot{H}^2}\\
  &\lesssim_{E_{max}} \eta.
 \end{split}
\end{equation}
Note that instead of using the pointwise evaluation of $u_h=\sum P_Mu_h$, we can replace $u_h$ by an arbitrary Schwartz function, get the bound, and then use density arguments to recover \eqref{FLIME-Proof-HighEst3}.
When $j=2$, we proceed as follows, using Sobolev's inequality, the Bernstein's properties \eqref{BernSobProp}, \eqref{FLIME-ProofPreliminaryEst1} and the estimate for $u_l$ in \eqref{FLIME-Proof-Step1},
\begin{equation}\label{FLIME-Proof-HighEst4}
\begin{split}
\Vert \mathcal{O}\left(u_h^2u_l\right)\Vert_{L^2L^\frac{16}{15}}&\lesssim \Vert u_l\Vert_{L^4 L^8}\Vert u_h\Vert_{L^4L^\frac{16}{7}}\Vert u_h\Vert_{L^\infty L^\frac{8}{3}}\\
&\lesssim \varepsilon\Vert u_l\Vert_{\dot{S}^2}\Vert \vert\nabla\vert^{-\frac{3}{2}}u_h\Vert_{L^2L^\frac{8}{3}}^\frac{1}{2}\Vert \vert\nabla\vert^\frac{3}{2}u_h\Vert_{L^\infty L^2}^\frac{1}{2}\\
&\lesssim \varepsilon^\frac{5}{2}\Vert u_h\Vert_{\dot{S}^{-\frac{3}{2}}}^\frac{1}{2}.
\end{split}
\end{equation}
When $j=1$, we proceed similarly to get
\begin{equation}\label{FLIME-Proof-HighEst5}
\begin{split}
\Vert \vert\nabla\vert^{-\frac{5}{2}}P_{\ge 1}\mathcal{O}\left(u_l^2u_h\right)\Vert_{L^2L^\frac{8}{5}}&\lesssim \Vert \mathcal{O}\left(u_l^2u_h\right)\Vert_{L^2L^\frac{8}{5}}\\
&\lesssim \Vert u_l\Vert_{L^4L^8}^2\Vert u_h\Vert_{L^\infty L^\frac{8}{3}}\\
&\lesssim \varepsilon^3,
\end{split}
\end{equation}
and finally,
\begin{equation}\label{FLIME-Proof-HighEst6}
 \begin{split}
  \Vert\vert\nabla\vert^{-\frac{5}{2}}P_{\ge 1}\vert u_l\vert^2u_l\Vert_{L^2L^\frac{8}{5}}&\lesssim \Vert \vert\nabla\vert \vert u_l\vert^2u_l\Vert_{L^2L^\frac{8}{5}}\\
  &\lesssim \Vert u_l\Vert_{\dot{S}^2}^3\\
  &\lesssim \varepsilon^3.
 \end{split}
\end{equation}
Combining \eqref{FLIME-Proof-HighEst1} and \eqref{FLIME-Proof-HighEst3}--\eqref{FLIME-Proof-HighEst6}, we get that
\begin{equation*}\label{FLIME-Proof-HighEst7}
\begin{split}
 \Vert u_h\Vert_{\dot{S}^{-\frac{3}{2}}}&\lesssim \varepsilon+\eta+\varepsilon^3+\varepsilon^\frac{5}{2}\Vert u_h\Vert_{\dot{S}^{-\frac{3}{2}}}^\frac{1}{2}\\
&\lesssim \eta.
\end{split}
\end{equation*}
This ends the proof of of \eqref{FLIME-Proof-Step1}.
As a consequence of conservation of energy, \eqref{FLIME-ProofPreliminaryEst1}, \eqref{FLIME-Proof-Step1} and Hardy-Littlewood-Sobolev's inequality, we get the following estimates on $J=J(2)$. Namely,
\begin{equation}\label{FLIME-Proof-XXXEstimates}
\begin{split}
&\Vert u_h\Vert_{L^\frac{5}{2}L^\frac{5}{2}}\lesssim_{E_{max}} \eta^\frac{4}{5}\hskip.1cm,\hskip.1cm
\Vert u_h\Vert_{L^3L^\frac{8}{3}}\lesssim_{E_{max}} \eta^\frac{2}{3}\hskip.1cm,\hskip.1cm
\Vert u_h\Vert_{L^\frac{9}{2}L^3}\lesssim_{E_{max}} \eta^\frac{4}{9}\hskip.1cm,\hskip.1cm\hbox{and}\\
&\Vert \vert u_h\vert^2\ast \vert x\vert^{-1}\Vert_{L^3L^{24}}\lesssim \Vert u_h\Vert_{L^6L^\frac{24}{11}}^2\lesssim_{E_{max}} \varepsilon^\frac{4}{3}\eta^\frac{2}{3}
\end{split}
\end{equation}
Now that we have good Strichartz control on the high and low frequencies, we can control the error terms arising in the frequency-localized interaction Morawetz estimates.
First, we treat the terms arising from the mass bracket. We claim that on $J=J(2)$, as defined above, we have that
\begin{equation}\label{FLIME-Proof-MBacketClaim}
\begin{split}
\int_J\int_{\mathbb{R}^{2n}}\{P_{\ge 1}\left(\vert u\vert^2u\right),u\}_m(t,y)\frac{(x-y)_j}{\vert x-y\vert}\{\partial_ju,u\}_m(t,x) dxdy&\lesssim \varepsilon^2\eta^2.
\end{split}
\end{equation}
Exploiting cancellations, we write
\begin{equation}\label{FLIME-Proof-MBacketEqt1}
\begin{split}
 \{P_{\ge 1}\left(\vert u\vert^2u\right),u_h\}_m=&\{P_{\ge 1}\left(\vert u\vert^2u-\vert u_h\vert^2u_h\right),u_h\}_m\\
 &-\{P_{< 1}\left(\vert u_h\vert^2u_h\right),u_h\}_m
+\{ \vert u_h\vert^2u_h,u_h\}_m.
\end{split}
\end{equation}
The last term in the right-hand side of \eqref{FLIME-Proof-MBacketEqt1} vanishes.
Using the Bernstein's properties \eqref{BernSobProp}, \eqref{FLIME-ProofPreliminaryEst1} and \eqref{FLIME-Proof-XXXEstimates}, we get that
\begin{equation}\label{FLIME-Proof-MBacketEqt2}
 \begin{split}
 & \left\vert\int_J\int_{\mathbb{R}^{2n}}\hbox{Im}\left(\partial_ku_h(x)\bar{u}_h(x)\right)\frac{(x-y)_k}{\vert x-y\vert}\{P_{< 1}\vert u_h\vert^2u_h,u_h\}_m(y)dxdydt\right\vert\\
&\lesssim \Vert u_h\Vert_{L^\infty L^2}\Vert\nabla u_h\Vert_{L^\infty L^2}\int_J\left\vert\left( P_{< 1}\vert u_h\vert^2u_h\right)u_h\right\vert dxdt\\
&\lesssim \varepsilon^2\Vert u_h\Vert_{L^3 L^\frac{8}{3}}\Vert P_{< 1}\vert u_h\vert^2u_h\Vert_{L^\frac{3}{2}L^\frac{8}{5}}\\
&\lesssim \varepsilon^2 \Vert u_h\Vert_{L^3 L^\frac{8}{3}}  \Vert \vert u_h\vert^2u_h\Vert_{L^\frac{3}{2}L^1}\\
&\lesssim \varepsilon^2 \Vert u_h\Vert_{L^3 L^\frac{8}{3}}^3\Vert u_h\Vert_{L^\infty L^4}\\
&\lesssim \varepsilon^2 \eta^2.\\
 \end{split}
\end{equation}
As for the first term in \eqref{FLIME-Proof-MBacketEqt1},
using \eqref{FLIME-ProofPreliminaryEst1}, we get that
\begin{equation}\label{FLIME-Proof-MBacketEqt3}
\begin{split}
&\left\vert \int_J\int_{\mathbb{R}^{2n}}\{\partial_ku_h,u_h\}_m(x)\frac{(x-y)_k}{\vert x-y\vert}\{ P_{\ge 1} \left(\vert u\vert^2u-\vert u_h\vert^2u_h\right),u_h\}_m(y) dxdydt
\right\vert\\
&
\lesssim \sum_{j=0}^2 \Vert u_h\Vert_{L^\infty H^1}^2\int_J\left\vert \left(P_{\ge 1}\mathcal{O}\left(u_h^ju_l^{3-j}\right)\right)u_h\right\vert dxdt
\end{split}
\end{equation}
and, using the Bernstein's properties \eqref{BernSobProp},  \eqref{FLIME-Proof-Step1}, and \eqref{FLIME-Proof-XXXEstimates}, we obtain
\begin{equation}\label{FLIME-Proof-MBacketEqt4}
\begin{split}
 \int_J\left\vert P_{\ge 1}\mathcal{O}\left(u_h^2u_l\right)u_h\right\vert dxdt
 &\lesssim \Vert u_h\Vert_{L^\frac{9}{2}L^3}\Vert u_h^2u_l\Vert_{L^\frac{9}{7}L^\frac{3}{2}}\\
 &\lesssim \Vert u_h\Vert_{L^\frac{9}{2}L^3}^3\Vert u_l\Vert_{L^3 L^\infty}\\
 &\lesssim \eta^\frac{4}{3}\varepsilon.
\end{split}
\end{equation}
Similarly,
\begin{equation}\label{FLIME-Proof-MBacketEqt5}
\begin{split}
 \int_J\left\vert P_{\ge 1} \mathcal{O} \left(u_h u_l^2\right)u_h\right\vert dxdt
 &\lesssim \Vert u_h\Vert_{L^3 L^\frac{8}{3}}^2\Vert u_l\Vert_{L^{6}L^8}^2\\
 &\lesssim \eta^\frac{4}{3}\varepsilon^2.
\end{split}
\end{equation}
In order to treat the last term, we remark that, in view of the Fourier support, if $M_1,M_2,M_3\le 1/8$, then $P_{\ge 1}\left(P_{M_1}uP_{M_2}uP_{M_3}u\right)=0$. Consequently, letting $c_M=M^2\Vert P_M u\Vert_{L^2L^4}$ and $d_M=M^2\Vert P_Mu\Vert_{L^\infty L^2}$, we get, using again the Bernstein's properties \eqref{BernSobProp}, \eqref{FLIME-ProofPreliminaryEst1} and \eqref{FLIME-Proof-Step1}, that
\begin{equation}\label{FLIME-Proof-MBacketEqt6}
\begin{split}
 &\int_J\left\vert \left(P_{\ge 1}\vert u_l\vert^2u_l\right)u_h\right\vert dxdt\\
&\lesssim \Vert u_h\Vert_{L^\infty L^2}\Vert P_{\ge 1}\sum_{1\ge M_1\ge M_2\ge M_3}P_{M_1}uP_{M_2}uP_{M_3}u\Vert_{L^1L^2}\\
&\lesssim \Vert u_h\Vert_{L^\infty L^2}\sum_{1\ge M_1\ge 1/8,M_1\ge M_2\ge M_3}\Vert P_{M_1}uP_{M_2}uP_{M_3}u\Vert_{L^1L^2}\\
&\lesssim \Vert u_h\Vert_{L^\infty L^2}\sum_{1\ge M_1\ge 1/8,M_1\ge M_2\ge M_3}
\Vert P_{M_1}\Vert_{L^2 L^4}\Vert P_{M_2}u\Vert_{L^2L^4}\Vert P_{M_3}u\Vert_{L^\infty L^\infty}\\
&\lesssim \Vert u_h\Vert_{L^\infty L^2}\left(\sum_{1\ge M\ge 1/8}\Vert P_{M}u\Vert_{L^2 L^4}\right)\sum_{1\ge M_2\ge M_3}c_{M_2}d_{M_3}\left(\frac{M_3}{M_2}\right)^2\\
&\lesssim \Vert u_h\Vert_{L^\infty L^2}\Vert u_l\Vert_{\dot{S}^2}^3\\
&\lesssim \varepsilon^4. 
\end{split}
\end{equation}
Combining \eqref{FLIME-Proof-MBacketEqt1}--\eqref{FLIME-Proof-MBacketEqt6},
we see that \eqref{FLIME-Proof-MBacketClaim} holds true.
Now, we turn to the last error term, which arises from the momentum bracket. We claim that on $J=J(2)$, we have that
\begin{equation}\label{FLIME-Proof-MomBacketClaim}
\begin{split}
 &\Big\vert\int_J\int_{\mathbb{R}^{2n}}\vert u_h(s,y)\vert^2\frac{(x-y)_j}{\vert x-y\vert}\{P_{\ge 1}\vert u\vert^2u,u_h\}_p^j(s,x)dxdyds\\
 &-\frac{1}{2}\int_{J}\int_{\mathbb{R}^{2n}}\frac{\vert u_h(s,y)\vert^2\vert u_h(s,x)\vert^4}{\vert x-y\vert}dxdyds\Big\vert\\
 &\lesssim \eta^2\left(\varepsilon^\frac{7}{3}\eta^{-\frac{2}{3}}+\varepsilon^2+\varepsilon^\frac{16}{3}\eta^{-\frac{4}{3}}\right).
\end{split}
\end{equation}
In order to prove \eqref{FLIME-Proof-MomBacketClaim}, we decompose
\begin{equation}\label{FLIME-Proof-MomBacketNewEqt1}
\begin{split}
\{P_{\ge 1}\vert u\vert^2u,u_h\}_p&=\{\vert u\vert^2u,u\}_p-\{\vert u_l\vert^2u_l,u_l\}_p-\{\left(\vert u\vert^2u-\vert u_l\vert^2u_l\right),u_l\}_p\\
&-\{P_{< 1}\vert u\vert^2u,u_h\}_p\\
&=-\frac{1}{2}\nabla\left(\vert u\vert^4-\vert u_l\vert^4\right)-\{\left(\vert u\vert^2u-\vert u_l\vert^2u_l\right),u_l\}_p\\
&-\{P_{< 1}\vert u\vert^2u,u_h\}_p.
\end{split}
\end{equation}
Besides, we remark that
\begin{equation}\label{FLIME-Proof-MomBacketEqt2}
\begin{split}
 \{ f,g\}_p&=\nabla\mathcal{O}\left(fg\right)-\mathcal{O}\left(f\nabla g\right).
\end{split}
\end{equation}
Now,
we estimate
\begin{equation}\label{FLIME-Proof-MomBacketNewEqt2}
\begin{split}
\mathcal{R}=\sum_{k=0}^{2}\int_J\int_{\mathbb{R}^{2n}}\vert u_h(s,y)\vert^2\frac{(x-y)_j}{\vert x-y\vert}\{\mathcal{O}\left(u_l^ku_h^{3-k}\right),u_l\}_p^j(s,x)dxdyds.
\end{split}
\end{equation}
The case $k=2$ is treated as follows, using \eqref{FLIME-Proof-MomBacketEqt2}. First
\begin{equation}\label{FLIME-Proof-MomBacketNewEqt3}
\begin{split}
&\left\vert\int_J\int_{\mathbb{R}^{2n}}\vert u_h(y)\vert^2\frac{(x-y)_j}{\vert x-y\vert}\mathcal{O}\left(u_hu_l^2\right)\partial_j u_l(x)dsdxdy\right\vert\\
&\lesssim \left\vert\int_J\int_{\mathbb{R}^n}\vert u_h(y)\vert^2\mathcal{O}\int_{\mathbb{R}^n}\left(\vert \nabla\vert^{-1}u_h\right)\left(\vert \nabla\vert\left(\frac{(x-y)_j}{\vert x-y\vert}\left(u_l^2\partial_ju_l\right)(x)\right)\right)dxdsdy\right\vert.
\end{split}
\end{equation}
Now, using the boundedness of the Riesz transform and the Bernstein's properties \eqref{BernSobProp}, we estimate for any $y\in\mathbb{R}^n$,
\begin{equation}\label{FLIME-Proof-MomBacketNewEqt4}
\begin{split}
&\Vert \vert\nabla\vert\left(\frac{(x-y)_j}{\vert x-y\vert}\left(u_l^2\partial_j u_l\right)(x)\right)\Vert_{L^\frac{8}{5}}\\
&\lesssim \Vert \nabla\left(\frac{(x-y)_j}{\vert x-y\vert}\left(u_l^2\partial_j u_l\right)(x)\right)\Vert_{L^\frac{8}{5}}\\
&\lesssim \Vert \vert x-y\vert^{-1}\mathfrak{1}_{\{\vert x-y\vert\le 1\}}\Vert_{L^2}\Vert u_l^2\Vert_{L^\infty}\Vert \nabla u_l\Vert_{L^8}\\
&+\Vert \vert x-y\vert^{-1}\mathfrak{1}_{\{\vert x-y\vert\ge 1\}}\Vert_{L^\infty}\Vert u_l^2\Vert_{L^4}\Vert\nabla u_l\Vert_{L^\frac{8}{3}}\\
&+\Vert \nabla u_l\Vert_{L^8}\Vert \partial_ju_l\Vert_{L^\frac{8}{3}}\Vert u_l\Vert_{L^8}+\Vert u_l^2\Vert_{L^4}\Vert \nabla^2u_l\Vert_{L^\frac{8}{3}}\\
&\lesssim \Vert u_l\Vert_{L^8}^2\Vert\nabla u_l\Vert_{L^\frac{8}{3}},
\end{split}
\end{equation}
where $\mathfrak{1}_E$ is the characteristic function of the set $E$.
Consequently, using the Bernstein's properties \eqref{BernSobProp}, \eqref{FLIME-ProofPreliminaryEst1}, \eqref{FLIME-Proof-Step1} and \eqref{FLIME-Proof-MomBacketNewEqt4}, we get that
\begin{equation}\label{FLIME-Proof-MomBacketNewEqt5}
\begin{split}
&\left\vert\int_J\int_{\mathbb{R}^n}\vert u_h(y)\vert^2\mathcal{O}\int_{\mathbb{R}^n}\left(\vert \nabla\vert^{-1}u_h\right)\left(\vert \nabla\vert\left(\frac{(x-y)_j}{\vert x-y\vert}\left(u_l^2\partial_ju_l\right)(x)\right)\right)dxdsdy\right\vert\\
&\lesssim \int_J\int_{\mathbb{R}^n}\vert u_h(y)\vert^2\Vert \vert\nabla\vert^{-1}u_h\Vert_{L^\frac{8}{5}}\Vert \vert\nabla\vert\left(\frac{(x-y)_j}{\vert x-y\vert}u_l^2\partial_j u_l\right)\Vert_{L^\frac{8}{5}}\\
&\lesssim \Vert u_h\Vert_{L^\infty L^2}^2\Vert \vert\nabla\vert^{-\frac{1}{2}}u_h\Vert_{L^2L^\frac{8}{3}}\Vert u_l\Vert_{L^4L^8}^2\Vert \nabla u_l\Vert_{L^\infty L^\frac{8}{3}}\\
&\lesssim \eta\varepsilon^5.
\end{split}
\end{equation}
Besides, integrating by parts and using \eqref{FLIME-Proof-Step1} and \eqref{FLIME-Proof-XXXEstimates}, we finish the analysis of the case $k=2$ as follows: 
\begin{equation}\label{FLIME-Proof-MomBacketNewEqt6}
\begin{split}
&\left\vert\int_J\int_{\mathbb{R}^{2n}}\vert u_h(y)\vert^2\vert x-y\vert^{-1}\vert\mathcal{O}\left(u_l^3u_h\right)(x)\vert dxdyds\right\vert\\
&\lesssim \Vert\vert u_h\vert^2\ast\vert x\vert^{-1}\Vert_{L^3 L^{24}}\Vert u_h\Vert_{L^\frac{5}{2}L^\frac{5}{2}}\Vert u_l\Vert_{L^\frac{45}{4}L^\frac{360}{67}}^3\\
&\lesssim \Vert u_h\Vert_{L^6 L^\frac{24}{11}}^2\Vert u_h\Vert_{L^\frac{5}{2}L^\frac{5}{2}}\Vert u_l\Vert_{\dot{S}^2}^3\\
&\lesssim \eta\varepsilon^4.
\end{split}
\end{equation}
The case $k=1$
is similar. First, with the Bernstein's properties \eqref{BernSobProp}, \eqref{FLIME-ProofPreliminaryEst1}, \eqref{FLIME-Proof-Step1} and \eqref{FLIME-Proof-XXXEstimates}, we obtain that
\begin{equation}\label{FLIME-Proof-MomBacketNewEqt7}
\begin{split}
&\left\vert\int_J\int_{\mathbb{R}^{2n}}\vert u_h(y)\vert^2\frac{(x-y)_j}{\vert x-y\vert}\mathcal{O}\left(u_h^2u_l\right)(x)\partial_j u_l(x)dsdxdy\right\vert\\
&\lesssim \Vert u_h\Vert_{L^\infty L^2}^2\Vert u_h^2\Vert_{L^\frac{3}{2}L^\frac{4}{3}}\Vert\nabla u_l\Vert_{L^3L^\frac{24}{5}}\Vert u_l\Vert_{L^\infty L^{24}}\\
&\lesssim \Vert u_h\Vert_{L^\infty L^2}^2\Vert u_l\Vert_{L^\infty L^4}\Vert u_l\Vert_{\dot{S}^2}\Vert u_h\Vert_{L^3L^\frac{8}{3}}^2\\
&\lesssim \varepsilon^4\eta^\frac{4}{3}
\end{split}
\end{equation}
and then,
\begin{equation}\label{FLIME-Proof-MomBacketNewEqt8}
\begin{split}
&\left\vert\int_J\int_{\mathbb{R}^{2n}}\vert u_h(y)\vert^2\vert x-y\vert^{-1}\mathcal{O}\left(u_h^2u_l^2\right)(x)dsdxdy\right\vert\\
&\lesssim \Vert \vert u_h\vert^2\ast\vert x\vert^{-1}\Vert_{L^3L^{24}}\Vert u_h\Vert_{L^\infty L^2}\Vert u_h\Vert_{L^\infty L^\frac{8}{3}}\Vert u_l\Vert_{L^3L^{24}}^2\\
&\lesssim \Vert u_h\Vert_{L^6L^\frac{24}{11}}^2\Vert u_h\Vert_{L^\infty H^1}^2\Vert u_l\Vert_{L^3L^{12}}^2\\
&\lesssim \varepsilon^\frac{16}{3}\eta^\frac{2}{3}.
\end{split}
\end{equation}
Finally for the case $k=0$, using the Bernstein's properties \eqref{BernSobProp}, \eqref{FLIME-ProofPreliminaryEst1}, \eqref{FLIME-Proof-Step1} and \eqref{FLIME-Proof-XXXEstimates}, we write that
\begin{equation}\label{FLIME-Proof-MomBacketNewEqt9}
\begin{split}
&\left\vert\int_J\int_{\mathbb{R}^{2n}}\vert u_h(y)\vert^2\frac{(x-y)_j}{\vert x-y\vert}\mathcal{O}\left(u_h^3\right)(x)\partial_j u_l(x)dsdxdy\right\vert\\
&\lesssim \Vert u_h\Vert_{L^\infty L^2}^2\Vert u_h^3\Vert_{L^\frac{3}{2}L^1}\Vert\nabla u_l\Vert_{L^3L^\infty}\\
&\lesssim \eta^\frac{4}{3}\varepsilon^3,
\end{split}
\end{equation}
and that
\begin{equation}\label{FLIME-Proof-MomBacketNewEqt10}
\begin{split}
&\left\vert\int_J\int_{\mathbb{R}^{2n}}\vert u_h(y)\vert^2\vert x-y\vert^{-1}\mathcal{O}\left(u_h^3u_l\right)(x)dsdxdy\right\vert\\
&\Vert \vert u_h\vert^2\ast\vert x\vert^{-1}\Vert_{L^3L^{24}}\Vert u_h\Vert_{L^3L^\frac{8}{3}}\Vert u_h\Vert_{L^\infty L^4}^2\Vert u_l\Vert_{L^3L^{12}}\\
&\lesssim \varepsilon^\frac{7}{3}\eta^\frac{4}{3}.
\end{split}
\end{equation}
This finishes the analysis of the second error term in the momentum bracket \eqref{FLIME-Proof-MomBacketNewEqt1}, namely $\mathcal{R}$.
Now we turn to the third error term arising from \eqref{FLIME-Proof-MomBacketNewEqt1}, i.e.
$$\tilde{\mathcal{R}}=\sum_{k=0}^3\int_J\int_{\mathbb{R}^{2n}}\vert u_h(s,y)\vert^2\frac{(x-y)_j}{\vert x-y\vert}\{P_{< 1}\mathcal{O}\left(u_h^ku_l^{3-k}\right),u_h\}_p^j(s,x)dxdyds.$$
We treat the term $k=0$ using \eqref{FLIME-Proof-MomBacketEqt2} as follows.
First, we get that
\begin{equation}\label{FLIME-Proof-MomBacket2NewEqt1}
\begin{split}
 &\left\vert\int_J\int_{\mathbb{R}^n}\vert u_h(y)\vert^2\int_{\mathbb{R}^n}u_h\frac{(x-y)_j}{\vert x-y\vert}\partial_j\left(P_{< 1} \mathcal{O}\left(u_l^3\right)\right)(x)dxdyds\right\vert\\
&=\left\vert\int_J\int_{\mathbb{R}^n}\vert u_h(y)\vert^2\int_{\mathbb{R}^n}\left(\vert\nabla\vert^{-1}u_h\right)\left(\vert\nabla\vert \left(\frac{(x-y)_j}{\vert x-y\vert}\partial_jP_{< 1} \mathcal{O}\left(u_l^3\right)\right)\right)dxdyds\right\vert\\
&\le \int_J\int_{\mathbb{R}^n}\vert u_h(y)\vert^2\Vert \vert\nabla\vert^{-1}u_h\Vert_{L^\frac{8}{3}}\Vert \vert\nabla\vert\left(\frac{(x-y)_j}{\vert x-y\vert}\partial_jP_{< 1}\mathcal{O}\left(u_l^3\right)\right)\Vert_{L^\frac{8}{5}}dsdy.
\end{split}
\end{equation}
Using the boundedness of the Riesz transform, we see that
\begin{equation}\label{FLIME-Proof-MomBacket2NewEqt2}
\begin{split}
&\Vert \vert\nabla\vert\left(\frac{(x-y)_j}{\vert x-y\vert}\partial_jP_{< 1}\mathcal{O}\left(u_l^3\right)\right)\Vert_{L^\frac{8}{5}}\\
&\lesssim \Vert\nabla\left(\frac{(x-y)_j}{\vert x-y\vert}\partial_jP_{< 1}\mathcal{O}\left(u_l^3\right)\right)\Vert_{L^\frac{8}{5}}\\
&\lesssim \Vert \mathfrak{1}_{\{\vert x-y\vert\le 1\}}\vert x-y\vert^{-1}\Vert_{L^2}\Vert \partial_ju_l\Vert_{L^8}\Vert u_l\Vert_{L^\infty}^2\\
&+\Vert \mathfrak{1}_{\{\vert x-y\vert\ge 1\}}\vert x-y\vert^{-1}\Vert_{L^\infty}\Vert \partial_ju_l\Vert_{L^8}\Vert u_l\Vert_{L^4}^2\\
&+\Vert \nabla\partial_ju_l\Vert_{L^8}\Vert u_l\Vert_{L^4}^2+\Vert \nabla u_l\Vert_{L^8}\Vert \nabla u_l\Vert_{L^4}\Vert u_l\Vert_{L^4}\\
&\lesssim \Vert u_l\Vert_{L^4}^2\Vert\nabla u_l\Vert_{L^8},
\end{split}
\end{equation}
and, consequently, using the Bernstein's properties \eqref{BernSobProp}, \eqref{FLIME-ProofPreliminaryEst1}, \eqref{FLIME-Proof-Step1} and \eqref{FLIME-Proof-MomBacket2NewEqt2} above, we obtain that
\begin{equation}\label{FLIME-Proof-MomBacket2NewEqt3}
\begin{split}
 &\left\vert\int_J\int_{\mathbb{R}^n}\vert u_h(y)\vert^2\left(\int_{\mathbb{R}^n}u_h(x)\frac{(x-y)_j}{\vert x-y\vert}\partial_j\left(P_{< 1} \vert u_l\vert^2u_l\right)(x)dx\right)dyds\right\vert\\
&\lesssim \int_J\int_{\mathbb{R}^n}\vert u_h(y)\vert^2\Vert \vert\nabla\vert^{-1}u_h\Vert_{L^\frac{8}{3}}
\Vert \nabla u_l\Vert_{L^8}\Vert u_l\Vert_{L^4}^2dyds\\
&\lesssim \Vert u_h\Vert_{L^\infty L^2}^2\Vert \vert\nabla\vert^{-\frac{1}{2}}u_h\Vert_{L^2L^\frac{8}{3}}\Vert \nabla u_l\Vert_{L^2L^8}\Vert u_l\Vert_{L^\infty L^4}^2\\
&\lesssim \eta\varepsilon^5.
\end{split}
\end{equation}
As for the other part, using \eqref{FLIME-Proof-Step1} and \eqref{FLIME-Proof-XXXEstimates}, we get that
\begin{equation}\label{FLIME-Proof-MomBacket2NewEqt4}
\begin{split}
& \left\vert\int_J\int_{\mathbb{R}^{2n}}\frac{\vert u_h(y)\vert^2}{\vert x-y\vert}P_{< 1}\mathcal{O}\left( u_l^3\right)(x)u_h(x)dxdyds\right\vert\\
&\lesssim \Vert \vert u_h\vert^2\ast \vert x\vert^{-1}\Vert_{L^3L^{24}}\Vert u_h\Vert_{L^3L^\frac{8}{3}}\Vert u_l\Vert_{L^3L^{12}}\Vert u_l\Vert_{L^\infty L^4}^2\\
&\lesssim \varepsilon^\frac{16}{3}\eta^\frac{4}{3}.
\end{split}
\end{equation}
Now, we treat the case $k=1$ using Bernstein property \eqref{BernSobProp}, \eqref{FLIME-ProofPreliminaryEst1}, \eqref{FLIME-Proof-Step1} and \eqref{FLIME-Proof-XXXEstimates} as follows. First we write that
\begin{equation}\label{FLIME-Proof-MomBacket2NewEqt5}
\begin{split}
 &\left\vert\int_J\int_{\mathbb{R}^{2n}} \vert u_h(y)\vert^2 \frac{(x-y)_j}{\vert x-y\vert}u_h(x)\partial_j\left(P_{< 1} \mathcal{O}\left(u_l^2u_h\right)\right)(x)dx dy ds\right\vert\\
&\lesssim \Vert u_h\Vert_{L^\infty L^2}^2\Vert u_h\Vert_{L^3L^\frac{8}{3}}\Vert u_l^2u_h\Vert_{L^\frac{3}{2}L^\frac{8}{5}}\\
&\lesssim \Vert u_h\Vert_{L^\infty L^2}^2\Vert u_h \Vert_{L^3 L^\frac{8}{3} }^2 \Vert u_l\Vert_{L^6L^8}^2 \\
&\lesssim \varepsilon^4\eta^\frac{4}{3},
\end{split}
\end{equation}
and then we write that
\begin{equation}\label{FLIME-Proof-MomBacket2NewEqt6}
\begin{split}
 &\left\vert\int_J\int_{\mathbb{R}^{2n}}\frac{\vert u_h(y)\vert^2}{\vert x-y\vert}u_h(x)P_{< 1} \mathcal{O}\left(u_l^2u_h\right)(x)dx dy ds\right\vert\\
&\lesssim \Vert\vert u_h\vert^2\ast\vert x\vert^{-1}\Vert_{L^3L^{24}}\Vert u_h\Vert_{L^3L^\frac{8}{3}}^2\Vert u_l\Vert_{L^\infty L^\frac{48}{5}}^2\\
&\lesssim \varepsilon^\frac{10}{3}\eta^2.
\end{split}
\end{equation}
When $k=2$, we use the Bernstein's properties \eqref{BernSobProp}, \eqref{FLIME-ProofPreliminaryEst1}, \eqref{FLIME-Proof-Step1}, and \eqref{FLIME-Proof-XXXEstimates} to get
\begin{equation}\label{FLIME-Proof-MomBacket2NewEqt7}
\begin{split}
 &\left\vert\int_J\int_{\mathbb{R}^{2n}} \vert u_h(y)\vert^2 \frac{(x-y)_j}{\vert x-y\vert}u_h(x)\partial_j\left(P_{< 1} \mathcal{O}\left(u_lu_h^2\right)\right)(x)dx dy ds\right\vert\\
&\lesssim \Vert u_h\Vert_{L^\infty L^2}^2\Vert u_h\Vert_{L^3L^\frac{8}{3}}\Vert \partial_jP_{< 1}\mathcal{O}\left(u_lu_h^2\right)\Vert_{L^\frac{3}{2}L^\frac{8}{5}}\\
&\lesssim \Vert u_h\Vert_{L^\infty L^2}^2\Vert u_h\Vert_{L^3L^\frac{8}{3}}\Vert u_h^2u_l\Vert_{L^\frac{3}{2}L^\frac{8}{6}}\\
&\lesssim \Vert u_h\Vert_{L^\infty L^2}^2\Vert u_h\Vert_{L^3L^\frac{8}{3}}^3\Vert u_l\Vert_{L^\infty L^\infty}\\
&\lesssim \varepsilon^3\eta^2,
\end {split}
\end{equation}
and
\begin{equation}\label{FLIME-Proof-MomBacket2NewEqt8}
\begin{split}
 &\left\vert\int_J\int_{\mathbb{R}^{2n}}\frac{\vert u_h(y)\vert^2}{\vert x-y\vert}u_h(x)P_{< 1}\mathcal{O}\left( u_lu_h^2\right)(x)dx dy ds\right\vert\\
&\lesssim \Vert \vert u_h\vert^2\ast\vert x\vert^{-1}\Vert_{L^3L^{24}}\Vert u_h\Vert_{L^3L^\frac{8}{3}}\Vert P_{< 1}\mathcal{O}\left( u_lu_h^2\right)\Vert_{L^3L^\frac{12}{7}}\\
&\lesssim \Vert u_h\Vert_{L^6L^\frac{24}{11}}^2\Vert u_h\Vert_{L^3L^\frac{8}{3}}\Vert P_{< 1}\mathcal{O}\left(u_lu_h^2\right)\Vert_{L^3L^1}\\
&\lesssim \Vert u_h\Vert_{L^6L^\frac{24}{11}}^2\Vert u_h\Vert_{L^3L^\frac{8}{3}}^2\Vert u_h\Vert_{L^\infty L^2}\Vert u_l\Vert_{L^\infty L^8}\\
&\lesssim \varepsilon^\frac{10}{3}\eta^2.
\end{split}
\end{equation}
Finally, the case $k=3$ is treated as follows using the Bernstein's properties \eqref{BernSobProp}, \eqref{FLIME-ProofPreliminaryEst1}, \eqref{FLIME-Proof-Step1}, and \eqref{FLIME-Proof-XXXEstimates}
\begin{equation}\label{FLIME-Proof-MomBacket2NewEqt9}
\begin{split}
 &\left\vert\int_J\int_{\mathbb{R}^{2n}} \vert u_h(y)\vert^2 \frac{(x-y)_j}{\vert x-y\vert}u_h(x)\partial_j\left(P_{< 1} \mathcal{O}\left(u_h^3\right)\right)(x)dx dy ds\right\vert\\
 &\lesssim \Vert u_h\Vert_{L^\infty L^2}^2\Vert u_h\Vert_{L^3L^\frac{8}{3}}\Vert \nabla P_{< 1}\mathcal{O}\left(u_h^3\right)\Vert_{L^\frac{3}{2}L^\frac{8}{5}}\\
&\lesssim \Vert u_h\Vert_{L^\infty L^2}^2\Vert u_h\Vert_{L^3L^\frac{8}{3}}\Vert u_h^3\Vert_{L^\frac{3}{2}L^1}\\
&\lesssim \Vert u_h\Vert_{L^\infty L^2}^2\Vert u_h\Vert_{L^3L^\frac{8}{3}}\Vert u_h\Vert_{L^\frac{9}{2} L^3}^3\\
&\lesssim \varepsilon^2\eta^2
\end{split}
\end{equation}
and, similarly,
\begin{equation}\label{FLIME-Proof-MomBacket2NewEqt10}
\begin{split}
 &\left\vert\int_J\int_{\mathbb{R}^{2n}}\frac{\vert u_h(y)\vert^2}{\vert x-y\vert}u_h(x)P_{< 1}\mathcal{O}\left( u_h^3\right)(x)dx dy ds\right\vert\\
&\lesssim \Vert \vert u_h\vert^2\ast \vert x\vert^{-1}\Vert_{L^3L^{24}}\Vert u_h\Vert_{L^3L^\frac{8}{3}}\Vert P_{< 1}\mathcal{O}\left(u_h^3\right)\Vert_{L^3L^\frac{12}{7}}\\
&\lesssim \Vert u_h\Vert_{L^6L^\frac{24}{11}}\Vert u_h\Vert_{L^3L^\frac{8}{3}}\Vert P_{< 1}\mathcal{O}\left(u_h^3\right)\Vert_{L^3L^1}\\
&\lesssim \Vert u_h\Vert_{L^6L^\frac{24}{11}}^2\Vert u_h\Vert_{L^3L^\frac{8}{3}}^2\Vert u_h\Vert_{L^\infty L^\frac{16}{5}}^2\\
&\lesssim \varepsilon^\frac{10}{3}\eta^2.
\end{split}
\end{equation}
This finishes the analysis of $\tilde{\mathcal{R}}$.
The first error term in \eqref{FLIME-Proof-MomBacketNewEqt1} is now easy to treat. Indeed, integrating by parts,
\begin{equation}\label{FLIME-Proof-MomBacket2NewEqt11}
\begin{split}
&\left\vert\int_J\int_{\mathbb{R}^{2n}}\vert u_h(s,y)\vert^2\frac{(x-y)}{\vert x-y\vert}\nabla\left(\vert u\vert^4-\vert u_l\vert^4-\vert u_h\vert^4\right)(s,x)dxdyds\right\vert\\
&\le\sum_{k=1}^3\mathcal{O}\int_J\int_{\mathbb{R}^{2n}}\frac{\vert u_h(s,y)\vert^2\vert u_h(s,x)\vert^{4-k}\vert u_l(s,x)\vert^k}{\vert x-y\vert}dxdyds\\
&\lesssim \eqref{FLIME-Proof-MomBacketNewEqt6}+ \eqref{FLIME-Proof-MomBacketNewEqt8}+\eqref{FLIME-Proof-MomBacketNewEqt10}\\
&\lesssim \varepsilon^\frac{7}{3}\eta^\frac{4}{3}.
\end{split}
\end{equation}
Finally, with \eqref{FLIME-Proof-MomBacketNewEqt1}--\eqref{FLIME-Proof-MomBacket2NewEqt11}, we obtain \eqref{FLIME-Proof-MomBacketClaim}.
As a consequence of \eqref{FLIME-IntMorEst1}, \eqref{FLIME-Proof-MBacketClaim} and \eqref{FLIME-Proof-MomBacketClaim} on $J=J(2)$, we have that
\begin{equation}\label{FLIME-PRoof-Fin}
 \begin{split}
  \Vert \vert\nabla\vert^{-\frac{3}{2}}\vert u_h\vert^2\Vert_{L^2L^2}^2\lesssim \eta^2\left(\varepsilon^\frac{7}{3}\eta^{-\frac{2}{3}}+\varepsilon^2+\varepsilon^\frac{16}{3}\eta^{-\frac{4}{3}}\right) \le \eta^2,
 \end{split}
\end{equation}
if $\varepsilon>0$ is sufficiently small, and $\eta>\varepsilon$. Letting $\eta=\varepsilon^\frac{1}{2}$, we obtain
$J(2)\subset J(1)$. Finally, $J(1)$ is a closed, open nonempty subset of $\mathbb{R}$.
Hence $J(1)=\mathbb{R}$, and this finishes the proof.
\end{proof}

It follows from H\"older's inequality that in the situation of Proposition \ref{FLIME-FLIMEProp}, one also has the
estimates \eqref{FLIME-Proof-XXXEstimates} with $\eta=\varepsilon$.

\section{The Soliton case}\label{Section-Soliton}

In this section, we deal with the first scenario in Theorem \ref{3S-H2CritThm3S}, namely the soliton case. We prove that the soliton scenario is inconsistent with the frequency-localized Morawtez interaction estimates developed in Section \ref{Section-IME} and compactness up to rescaling.

\begin{proposition}\label{Sol-NoSolitonProp}
 Let $u\in C(\mathbb{R},\dot{H}^2)$ be a solution of \eqref{CubicFOS} such that $K=\{u(t):t\in\mathbb{R}\}$ is precompact in $\dot{H}^2$ up to translation. If $n=8$, then $u=0$. In particular the soliton scenario in Theorem \ref{3S-H2CritThm3S} does not hold true.
\end{proposition}

\begin{proof}
 Let $u\in C(\mathbb{R},\dot{H}^2)$ be a solution of \eqref{CubicFOS} of energy $E(u)>0$ such that 
 $K=\{g_{(1,y(t))}u(t):t\in\mathbb{R}\}$ is precompact in $\dot{H}^2$. In particular we can apply Proposition \ref{FLIME-FLIMEProp} with $\varepsilon>0$ and deduce that
\begin{equation}\label{Sol-ConsOfIME}
 \Vert \vert\nabla\vert^{-\frac{3}{4}}P_{\ge 1}u\Vert_{L^4(\mathbb{R},L^4)}\lesssim 1.
\end{equation}
Independently, by \eqref{FLIMEPropEst2}, we know that, for all $t$,
\begin{equation}\label{SolConsOfScal}
 \Vert P_{\ge 1}u(t)\Vert_{\dot{H}^2}^2\gtrsim_{E(u)} E(u)-\varepsilon^2>0,
\end{equation}
if $\varepsilon$ is sufficiently small. Then \eqref{SolConsOfScal} implies that for all 
$v$ in the $\dot{H}^2$-closure of $K$, $P_{\ge 1}v\ne 0$.
Since $K$ is precompact in $\dot{H}^2$ and the mapping $v\mapsto \Vert \vert\nabla\vert^{-\frac{3}{4}}P_{\ge 1}v\Vert_{L^4}$ is 
continuous on $\dot{H}^2$, we get that there exists $\kappa>0$ such that
\begin{equation}\label{FLIME-ConsOFComp}
\forall v\in K,\hskip.1cm\Vert\vert\nabla\vert^{-\frac{3}{4}}P_{\ge 1}v\Vert_{L^4}\ge\kappa.
\end{equation}
Now, \eqref{Sol-ConsOfIME} and \eqref{FLIME-ConsOFComp} imply that
\begin{equation}\label{Sol-ConsOfCompactness}
\kappa^4 t\lesssim \Vert \vert\nabla\vert^{-\frac{3}{4}}u_h\Vert_{L^4([0,t],L^4)}^4\lesssim 1.
\end{equation}
Letting $t\to +\infty$, we get a contradiction in \eqref{Sol-ConsOfCompactness}.
This finishes the proof of Proposition \ref{Sol-NoSolitonProp}.
\end{proof}

\section{The Low-to-high cascade}\label{Section-Cascade}

Now, we are ready to deal with the last scenario, and to exclude the case of a low-to-high cascade solution.
In order to do so, we use the estimates coming from the frequency-localized interaction Morawetz estimates developed in Section \ref{Section-IME} to control the action of the high-frequency part of $u$. Then the low-frequency part obeys an analogue of \eqref{CubicFOS} with initial data arbitrarily small. Hence one can make its $\dot{S}^2$-norm small, depending on the frequency, so as to prove that it is in fact small in $L^2$. Then the solution is an $H^2$ solution, and conservation of mass gives a contradiction. More precisely, we prove the following proposition.
\begin{proposition}\label{Cas-L2Regprop}
 Let $u\in C(I,\dot{H}^2)$ be a maximal lifespan solution of \eqref{CubicFOS} such that $K=\{g_{(h(t),x(t))}u(t):t\in I\}$ is precompact in $\dot{H}^2$ for some functions $h,x$ such that $h(t)\le h(0)=1$, and
\begin{equation}\label{Ca-HLCas}
 \liminf_{t\to \sup I}h(t)=0,
\end{equation}
then if $n=8$, we have that $u=0$. In particular, the low-to-high cascade scenario does not hold true.
\end{proposition}

\begin{proof}
 Let $u$ be as above. Applying Proposition \ref{SS-NoSelfSimilarSolutionProp}, we see that $I=\mathbb{R}$, and since $h\le 1$, given $\varepsilon>0$, we can apply Proposition \ref{FLIME-FLIMEProp} to get that \eqref{FLIMEPropEst2} holds true. We may also suppose that \eqref{FLIME-Proof-Step1} holds true.
As a first step in the proof, we claim that if $\varepsilon>0$ is suficiently small, the following holds true for all dyadic number $M\le 1$:
\begin{equation}\label{Cas-GainOfDerivative1}
 \Vert P_{\le M}u\Vert_{\dot{S}^2}\lesssim M^3.
\end{equation}
Fix $M_0$, a dyadic number, let $m=M_0^{10}$ and let $\kappa>0$ to be chosen later. 
Since we know that \eqref{Ca-HLCas} holds true and that $K$ is precompact, using \eqref{Cas-GainOfDerivative2} we get that there exists $t_0>0$ such that
\begin{equation}\label{Cas-GainOfDerivative3}
 \Vert P_{\le 1}u(t_0)\Vert_{\dot{H}^2}\le \kappa^2 m.
\end{equation}
We claim that for any $C>0$, if $\kappa$ is sufficiently small, independently of $m$,
then we have that, for all dyadic numbers $M\in [m, 1]$,
\begin{equation}\label{Cas-GainOfDerivative4}
 \Vert P_{\le M}u\Vert_{\dot{S}^2(J)}\le \kappa C\left(m+M^3\right)
\end{equation}
when $J$ is small and $t_0\in J$. Indeed, using the Bernstein's properties \eqref{BernSobProp}, we get that, in $J$,
\begin{equation*}\label{Cas-GainOfDerivative5}
 \begin{split}
  \Vert P_{\le M}u\Vert_{\dot{S}^2}^2 &\lesssim \sum_{N\le M}N^4\Vert P_Nu\Vert_{L^\infty L^2}^2+\sum_{N\le M}N^6\Vert P_Nu\Vert_{L^2 L^\frac{8}{3}}^2\\
&\lesssim \sum_{N\le M}N^4\Vert P_Nu(t_0)\Vert_{L^2}^2+\vert J\vert^2\sum_{N\le M}N^4\Vert\partial_tP_N\Vert_{L^\infty L^2}^2\\
&+\vert J\vert\sum_{N\le M}N^6\Vert P_Nu\Vert_{L^\infty L^\frac{8}{3}}^2\\
&\lesssim \kappa^4 m^2+\vert J\vert^2\sum_{N\le M}N^4\Vert P_N\Delta^4 u\Vert_{L^\infty L^2}^2\\
&+\vert J\vert^2\sum_{N\le M}N^4\Vert P_N\left(\vert u\vert^2u\right)\Vert_{L^\infty L^2}^2+\vert J\vert\sum_{N\le M}N^8\Vert P_Nu\Vert_{L^\infty L^2}^2\\
&\lesssim_{E(u)} \kappa^4 m^2+M^8\vert J\vert^2+\vert J\vert^2\sum_{N\le M}N^8\Vert \vert u\vert^2u\Vert_{L^\infty L^\frac{4}{3}}^2 +\vert J\vert M^4\\
&\lesssim_{E(u)}\kappa^4 m^2+M^8\vert J\vert^2+M^4\vert J\vert,\\
 \end{split}
\end{equation*}
and if $\vert J\vert\lesssim_{E(u),C}\kappa$, then \eqref{Cas-GainOfDerivative4} holds true. Now, let $J(C)$ be the maximum interval containing $t_0$ on which \eqref{Cas-GainOfDerivative4} holds true for the constant $C>0$. We prove that $J(2)\subset J(1)$ if $\kappa$ and $\varepsilon$ are chosen sufficiently small, independently of $m$.
Indeed, let 
\begin{equation*}
u_{vlow}=P_{\le m}u\hskip.1cm,\hskip.1cm\hbox{and}\hskip.1cm u_{med}=P_{m<\cdot< 1}u.
\end{equation*}
In the following, all time integrals are taken on $J=J(2)$. Applying Strichartz estimates \eqref{Not-SE-StrichartzEstimatesWithGainII}, we get that
\begin{equation}\label{Cas-GainOfDerivative6}
 \begin{split}
  &\Vert P_{\le M}u\Vert_{\dot{S}^2}\\
  &\lesssim \Vert P_{\le M}u(t_0)\Vert_{\dot{H}^2}+\Vert\nabla P_{\le M}\vert u_{vlow}\vert^2 u_{vlow}\Vert_{L^2L^\frac{8}{5}}\\
&+\Vert \nabla P_{\le M}\left(\vert u\vert^2u-\vert u_{vlow}\vert^2u_{vlow}\right)\Vert_{L^2L^\frac{8}{5}}\\
&\lesssim \kappa m +\Vert P_{\le m}u\Vert_{\dot{S}^2}^3+M\Vert \tilde{P}_{\le M}\left(\vert u_{vlow}\vert^2 \vert u_{med}\vert+\vert u_{vlow}\vert^2 \vert u_h\vert\right)\Vert_{L^2L^\frac{8}{5}}\\
&+M\Vert \tilde{P}_{\le M}\vert u_{med}\vert^3\Vert_{L^2L^\frac{8}{5}}+M\Vert \tilde{P}_{\le M}\vert u_h\vert^3\Vert_{L^2L^\frac{8}{5}},
 \end{split}
\end{equation}
where $\tilde{P}_{\le M}$ is the convolution operator whose kernel is
\begin{equation*}
\tilde{k}(x)=M^{8}\hat{\psi}(Mx)^2,
\end{equation*}
where $\psi$ is as in \eqref{DefLitPalOp}.  We remark that $\tilde{P}_{\le M}$ has nonnegative kernel and satisfies estimates similar to those of $P_{\le M}$. In particular, \eqref{BernSobProp} holds true with $\tilde{P}_{\le M}$ in place of $P_{\le M}$.
By assumption we have that
\begin{equation}\label{Cas-GainOfDerivative7}
\begin{split}
 \Vert P_{\le m}u\Vert_{\dot{S}^2}^3\le \left(4\kappa\right)^3m^3.
\end{split}
\end{equation}
Besides, using the Bernstein's properties \eqref{BernSobProp}, and the assumption on $J$, we get that
\begin{equation}\label{Cas-GainOfDerivative8}
 \begin{split}
  M\Vert \tilde{P}_{\le M}\vert u_{vlow}^2u_{med}\vert\Vert_{L^2L^\frac{8}{5}}
&\lesssim M\Vert u_{vlow}\Vert_{L^\infty L^4}^2\Vert u_{med}\Vert_{L^2L^8}\\
&\lesssim M\left(4\kappa m\right)^2\left(\sum_{m<N< 1} N^{-1}\Vert\vert\nabla\vert P_Nu\Vert_{L^2L^8}\right)\\
&\lesssim M\left(\kappa m\right)^2\left(\sum_{m<N< 1} N^{-1}\Vert P_{\le 2N}u\Vert_{\dot{S}^2}\right)\\
&\lesssim Mm^2\kappa^3\left(\sum_{m<N< 1}N^{-1}\left(m+N^3\right)\right)\\
&\lesssim m^2\kappa^3M
 \end{split}
\end{equation}
and, similarly, using the Bernstein's properties, \eqref{BernSobProp} and \eqref{FLIME-Proof-XXXEstimates}, we have that
\begin{equation}\label{Cas-GainOfDerivative9}
 \begin{split}
  M\Vert \tilde{P}_{\le M}\vert u_{vlow}^2u_{h}\vert\Vert_{L^2L^\frac{8}{5}}
&\lesssim M^2\Vert u_{vlow}^2u_h\Vert_{L^2L^\frac{8}{6}}\\
&\lesssim M^2\Vert u_{vlow}\Vert_{L^4 L^8}^2\Vert u_h\Vert_{L^\infty L^2}\\
&\lesssim \kappa^2 M^2m^2\varepsilon.
 \end{split}
\end{equation}
Independently, using the Bernstein's properties \eqref{BernSobProp} and the definition of $J$, we get that
\begin{equation}\label{Cas-GainOfDerivative10}
\begin{split}
 M\Vert \tilde{P}_{\le M}\vert u_{med}\vert^3\Vert_{L^2L^\frac{8}{5}}
&\lesssim M^3 \Vert u_{med}^3\Vert_{L^2L^\frac{8}{7}}\\
&\lesssim M^3\left(\sum_{m<N< 1}N^{-1}\Vert \nabla P_Nu\Vert_{L^6L^\frac{24}{7}}\right)^3\\
&\lesssim M^3\left(\sum_{m<N< 1}N^{-1}\Vert P_{\le 2N}\Vert_{\dot{S}^2}\right)^3\\
&\lesssim M^3\left(2\kappa\sum_{m<N\le 1}N^{-1}m+N^2\right)^3\\
&\lesssim \left( 2\kappa\right)^3M^3,
\end{split}
\end{equation}
and, using again the Bernstein's properties \eqref{BernSobProp} and \eqref{FLIMEPropEst2}, we obtain that
\begin{equation}\label{Cas-GainOfDerivative11}
\begin{split}
 M\Vert\tilde{P}_{\le M}\vert u_h\vert^3\Vert_{L^2L^\frac{8}{5}}&\lesssim M^4\Vert u_h^3\Vert_{L^2L^1}\\
&\lesssim M^4\Vert\vert\nabla\vert^{-\frac{1}{2}}u_h\Vert_{L^2L^\frac{8}{3}}\Vert\vert\nabla\vert^\frac{7}{4} u_h\Vert_{L^\infty L^2}^2\\
& \lesssim M^4\varepsilon^3.
\end{split}
\end{equation}
Finally, with \eqref{Cas-GainOfDerivative2}--\eqref{Cas-GainOfDerivative11}, we get, if $\kappa=\varepsilon$ and $\varepsilon$ is sufficiently small, that there holds that
\begin{equation}\label{Cas-GainOfDerivative12}
\begin{split}
\Vert P_{\le M}u\Vert_{\dot{S}^2}\le \kappa\left( m+M^3\right).
\end{split}
\end{equation}
In particular, $J(2)\subset J(1)$. Consequently, $J(1)$ is a closed, open nonempty subset of $\mathbb{R}$. Hence $J(1)=\mathbb{R}$. Then \eqref{Cas-GainOfDerivative12} gives \eqref{Cas-GainOfDerivative1} for $M\in (M_0^{10},1)$. Since $M_0$ can be chosen arbitrarily small, we get \eqref{Cas-GainOfDerivative1} for all $M\le 1$.
Now, we finish the proof of Proposition \ref{Cas-L2Regprop}. A consequence of \eqref{Cas-GainOfDerivative1} is that $u\in L^\infty L^2$. Indeed, by the Bernstein's properties \eqref{BernSobProp}, $P_{\ge M}u\in L^\infty L^2$ for any dyadic $M$, and using \eqref{Cas-GainOfDerivative1}, we get that, when $M\le 1$,
\begin{equation}\label{Cas-EndOfProof1}
 \begin{split}
  \Vert P_{\le M}u\Vert_{L^\infty L^2}&\le \sum_{N\le M}\Vert P_Nu\Vert_{L^\infty L^2}\\
&\lesssim \sum_{N\le M} N^{-2}\Vert P_{\le N}u\Vert_{\dot{S}^2}\\
&\lesssim \sum_{N\le M} N\lesssim M.
 \end{split}
\end{equation}
Now, let $M>0$ be an arbitrarily small dyadic number.
Since \eqref{Ca-HLCas} holds true, and since $K$ is precompact in $\dot{H}^2$, we can find $t_0$ such that 
\begin{equation}\label{Cas-EndOfProof2}
\begin{split} 
\Vert P_{M<\cdot \le M^{-1}}u(t_0)\Vert_{L^2}&\le M^{-2}\Vert P_{M<\cdot \le M^{-1}}u(t_0)\Vert_{\dot{H}^2}\\
&\le M^{-2}\Vert P_{Mh(t_0)<\cdot\le M^{-1}h(t_0)}\left(g(t_0)u(t_0)\right)\Vert_{\dot{H}^2}\\
&\le M.
\end{split}
\end{equation}
Using conservation of mass, the Bernstein's properties \eqref{BernSobProp}, \eqref{Cas-EndOfProof1} and \eqref{Cas-EndOfProof2}, we deduce that
\begin{equation}\label{Cas-EndOfProof3}
\begin{split}
\Vert u(0)\Vert_{L^2}&=\Vert u(t_0)\Vert_{L^2}\\
&\le \Vert P_{> M^{-1}}u(t_0)\Vert_{L^2}+\Vert P_{M<\cdot\le M^{-1}}u(t_0)\Vert_{L^2}+\Vert P_{\le M}u\Vert_{L^\infty L^2}\\
&\le M^2E(u)^\frac{1}{2}+2M.
\end{split}
\end{equation}
Since $M$ is arbitrary, we get that $u(0)=0$. This concludes the proof of Proposition \ref{Cas-L2Regprop}.
\end{proof}

\section{Analiticity of the flow map and scattering}\label{Section-Scattering}

In view of Theorem \ref{IP-IPTheorem} and Corollary \ref{3S-Cor}, we can
finish the proof of the first assertions in Theorem \ref{MainThm} with Proposition \ref{Scat-Anal} below.

\begin{proposition}\label{Scat-Anal}
Let $n\le 8$. Then, for any $t>0$, the mapping $u_0\mapsto u(t)$, from $H^2$ into $H^2$, is analytic.
\end{proposition}

\begin{proof}
We follow arguments developed in Pausader and Strauss \cite{PauStr} for the fourth-order wave equation.
We use the implicit function theorem.
In case $1\le n\le 3$, the global bound on the energy gives a global bound on the $L^\infty$-norm of $u$, and hence, the nonlinear term is lipschitz. In this case the problem can be solved with basic arguments. Now we treat the case $n\ge 4$. We divide $[0,t]=\cup_{j=1}^k I_j$ into subintervals $I_j=[a_j,a_{j+1}]$ such that
\begin{equation}\label{Scat-CondForAnal}
\Vert \nabla u\Vert_{L^\frac{n+4}{2}(I,L^\frac{n(n+4)}{3n+4})}\le \delta.
\end{equation}
First, if $I=I_0=[0,a_1]$, we consider the mapping
\begin{equation*}
\mathcal{T}_I:H^2\times \dot{S}^0(I)\cap\dot{S}^2(I)\to H^2\times\dot{S}^0(I)\cap \dot{S}^2(I)
\end{equation*}
defined by
\begin{equation*}\label{Scat-DefOfT}
\mathcal{T}\left(u_0,v\right)=\left(u_0,t\mapsto e^{it\Delta^2}u_0+i\int_0^te^{i(t-s)\Delta^2}\vert v\vert^2v(s)ds\right).
\end{equation*}
The map $\mathcal{T}$ is well defined thanks to the Strichartz estimates \eqref{Not-SE-StrichartzEstimatesWithGainII}. It is clearly analytic, and $u\in C(I,H^2)$ is a solution of \eqref{CubicFOS} if and only if $\mathcal{T}\left(u(0),u\right)=\left(u(0),u\right)$.
An application of Strichartz estimates gives that, if $\delta$ in \eqref{Scat-CondForAnal}
is sufficiently small, then
\begin{equation*}
\Vert D_2\mathcal{T}\left(u(0),u\right)\Vert_{\dot{S}^0\cap\dot{S}^2\to\dot{S}^0\cap\dot{S}^2}<1,
\end{equation*}
where $D_2$ denotes derivation with respect to the second argument.
Consequently, $D_2\left(I-\mathcal{T}\right)\left(u(0),u\right)$ is invertible, and the implicit function theorem ensures that $u_0\mapsto u_{\vert I}$ is analytic. In particular, $u_0\mapsto u(a_1)$, from $H^2$ into $H^2$, is analytic. By finite induction, we get that $u_0\mapsto u(t)$ is analytic.
\end{proof}

Now, we turn to the proof of the scattering assertion of Theorem \ref{MainThm}. The statement
is an easy consequence of Propositions \ref{LastProp} and \ref{ScatProp2} below.

\begin{proposition}\label{LastProp}
Let $5\le n\le 8$. For any $u^+\in H^2$, respectively $u^-\in
H^2$, there exists a unique $u\in C(\mathbb{R},H^2)$,
solution of \eqref{CubicFOS} such that
\begin{equation}\label{GE-ScatteringStatement}
  \Vert u(t)-e^{it\Delta^2}u^\pm\Vert_{H^2}\to 0
 \end{equation}
 as $t\to\pm\infty$. Besides, we have that
 \begin{equation}\label{ScatEnergies}
 \begin{split}
 M(u(0))&=M(u^\pm)\hskip.1cm,\hskip.1cm\text{and}\\
 2E(u(0))&=\Vert u^\pm\Vert_{\dot{H}^2}^2.
 \end{split}
 \end{equation}
This defines two mappings $\mathcal{W}_\pm:u^\pm\mapsto u(0)$ from $H^2$ into $H^2$, and $\mathcal{W}_+$ and $\mathcal{W}_-$
are continuous in $H^2$.
\end{proposition}

\begin{proof}
By time reversal symmetry, we need only to prove Proposition
\ref{LastProp} for $u^+$. Let
$\omega(t)=e^{it\Delta^2}u^+$. Then by the
Strichartz estimates \eqref{Not-SE-StrichartzEstimatesWithGainII},
$\omega\in \dot{S}^0(\mathbb{R})\cap \dot{S}^2(\mathbb{R})$
and, given $\delta>0$, there exists
$T_\delta$ such that, on $I=[T_\delta,+\infty)$, \eqref{Scat-CondForAnal} holds true with $\omega$ instead of $u$. 
For $u\in \dot{S}^0(I)\cap\dot{S}^2(I)$, we define
\begin{equation}\label{Scat-DefPhi}
\Phi(u)(t)=\omega(t)-i\int_{t}^\infty
e^{i(t-s)\Delta^2}\vert
u(s)\vert^2u(s)ds.
\end{equation}
For $\delta$ sufficiently small, $\Phi$ defines a contraction
mapping on the set
\begin{equation*}
\begin{split}
X_{T_\delta}=&\{u\in \dot{S}^0(I)\cap
\dot{S}^2(I);\Vert
\nabla u\Vert_{L^\frac{n+4}{2}(I,L^\frac{n(n+4)}{3n+4})}\le 2\delta,\\
&\Vert
u\Vert_{\dot{S}^0(I)}+\Vert u\Vert_{\dot{S}^2(I)}\lesssim
\Vert u^+\Vert_{H^2}\}\hskip.1cm,
\end{split}
\end{equation*}
equipped with the
$\dot{S}^0(I)$-norm.
Thus $\Phi$ admits a unique fixed point $u$. We
observe that
\begin{equation*}\label{LastEqt}
u(T_\delta+t)=e^{it\Delta^2}u(T_\delta)+i\int_{T_\delta}^{T_\delta+t}e^{i(t-s)\Delta^2}\vert
u(s)\vert^2u(s)ds
\end{equation*}
in $H^2$. Consequently, $u$ solves \eqref{CubicFOS}
on $I=[T_\delta,+\infty)$. Hence, using the first part of Theorem \ref{MainThm}, $u$ can be extended for all times $t\in\mathbb{R}$. Now,
\eqref{GE-ScatteringStatement} follows from \eqref{Scat-DefPhi}
and the boundedness of $u$ in $\dot{S}^2$ and
$\dot{S}^0$-norms. Uniqueness
follows from the fact that any solution of
\eqref{CubicFOS} has a restriction in $X_T$ for some $T\ge
T_\delta$, and uniqueness of the fixed point of $\Phi$ in such
spaces. The continuity statements are easy adaptations of the proof of local well-posedness, see Pausader \cite{Pau1}. The first equality in \eqref{ScatEnergies} follows from conservation of Mass and convergence in $L^2$. For the second, we remark that since $\omega\in \dot{S}^0(\mathbb{R})$ there exists a sequence of times $t_k\to +\infty$ such that $\Vert \omega(t_k)\Vert_{L^4}\to 0$. Then, using conservation of energy, we compute
 \begin{equation*}
 \begin{split}
 2E(u(0))&=2E(u(t_k))\\
 &=2E(\omega(t_k))+o(1)\\
 &=\Vert \omega(t_k)\Vert_{\dot{H}^2}^2+o(1)=\Vert u^+\Vert_{\dot{H}^2}^2+o(1),
 \end{split}
 \end{equation*}
 and letting $k\to +\infty$ we get that the second equation in \eqref{ScatEnergies} holds true.
This finishes the proof of Proposition \ref{LastProp}.
\end{proof}

\begin{proposition}\label{ScatProp2}
 Let $5\le n\le 8$. Given any solution $u\in C(\mathbb{R},H^2)$ of \eqref{CubicFOS}, there exist $u^\pm\in H^2$ such that
 \eqref{GE-ScatteringStatement} holds true. In particular $\mathcal{W}_\pm$ are homeomorphisms of $H^2$.
\end{proposition}

\begin{proof}

In case $5\le n\le 7$, the equation is subcritical, and standard developments using the decay properties of the linear propagator, conservation of mass and the usual Morawetz estimates, 
give that for any solution $u\in C(\mathbb{R},H^2)$ of \eqref{CubicFOS}, there exists $C>0$ such that
\begin{equation*}
\Vert u\Vert_{L^4(\mathbb{R},L^4)}\le C.
\end{equation*}
On such an assertion we refer to Cazenave \cite{Caz} or Lin and Strauss \cite{LinStr} for the second order case, 
and to Pausader \cite{Pau1} for the classical Morawetz estimates in the case of the fourth-order Schr\"odinger equation. 
Consequently, applying Strichartz estimates, we get that
\begin{equation}\label{Scat-ScatBound}
\Vert u\Vert_{\dot{S}^0(\mathbb{R})}+\Vert u\Vert_{\dot{S}^2(\mathbb{R})}\lesssim_u 1.
\end{equation}
In case $n=8$, as a consequence of Corollary \ref{3S-Cor}, we get that any nonlinear solution $u$ satisfies
\begin{equation*}
\Vert u\Vert_{Z(\mathbb{R})}\lesssim_{E(u)}1.
\end{equation*}
Using the work in Pausader \cite[Proposition $2.6$]{Pau1}, we then get that \eqref{Scat-ScatBound} holds true also when $n=8$.
 Since $e^{it\Delta^2}$ is an
 isometry on $H^2$, \eqref{GE-ScatteringStatement} is
 equivalent to proving that there exists $u^+\in H^2$ such that
 \begin{equation}\label{GE-ProofOfScatteringPropEstimate1}
  \Vert e^{-it\Delta^2}u(t)-u^+\Vert_{H^2}\to 0
 \end{equation}
 as $t\to+\infty$. Now we prove that $e^{-it\Delta^2}u(t)$
 satisfies a Cauchy criterion. We note that
 Duhamel's formula gives that
 \begin{equation}\label{GE-ProofOfScatteringPropEstimate2}
  e^{-it_1\Delta^2}u(t_1)-e^{-it_0\Delta^2}u(t_0)
  =i\int_{t_0}^{t_1}e^{-is\Delta^2}\vert
  u(s)\vert^2u(s)ds.
 \end{equation}
By duality, \eqref{Not-SE-StrichartzEstimatesWithGainII} gives that for any $s\in [0,2],$ and any $h\in\dot{\bar{S}}(\mathbb{R})$, we have that
\begin{equation}\label{Scat-StricScat}
\Vert\int_\mathbb{R}e^{-it\Delta^2}h(t)dt\Vert_{\dot{H}^s}\lesssim \Vert h\Vert_{\dot{\bar{S}}^s(\mathbb{R})}.
\end{equation}
Now, \eqref{Scat-ScatBound} and \eqref{Scat-StricScat} give that the right hand side in $\eqref{GE-ProofOfScatteringPropEstimate2}$ is like $o(1)$
 in $H^2$ as $t_0,t_1\to+\infty$. In particular,
 $e^{-it\Delta^2}u(t)$ satisfies a Cauchy
 criterion, and there exists $u^+\in H^2$ such that \eqref{GE-ProofOfScatteringPropEstimate1} holds true.
 We also get that
 \begin{equation}\label{GE-ScatteringForm}
 u^+=u_0+i\int_0^\infty e^{-is\Delta^2}\vert
 u(s)\vert^2u(s)ds,
 \end{equation}
 and $u^+$ is unique. The continuity statements are easy adaptations of the proof of local well-posedness, see Pausader \cite{Pau1}.
 Now, by uniqueness, we clearly have that $u(0)=\mathcal{W}_+(u^+)$, so that $\mathcal{W}_+$ is an homeomorphism. This ends the proof of Proposition \ref{ScatProp2}.
\end{proof}

\begin{proof}[Proof of the scattering in Theorem \ref{MainThm}]
 Applying Propositions \ref{LastProp} and \ref{ScatProp2}, we see that the scattering operator $S=\mathcal{W}_+\circ\mathcal{W}_-^{-1}$ is an homeomorphism from $H^2$ into $H^2$. 
 Using \eqref{Scat-DefPhi} and \eqref{GE-ScatteringForm}, and adapting slightly the proof of Proposition \ref{Scat-Anal}, we easily see that $S$ is analytic. This ends the proof of 
the scattering part in Theorem \ref{MainThm}.
\end{proof}

\medskip\noindent{\small ACKNOWLEDGEMENT:} The author expresses his deep thanks to Emmanuel Hebey for his
constant support and for stimulating discussions during the preparation of this work.

\end{document}